\documentclass[hidelinks,onefignum,onetabnum]{siamart250211}
\usepackage{lipsum}
\usepackage{amssymb,amsfonts,amstext,amsmath}
\usepackage{graphicx,geometry}
\usepackage{epstopdf}
\usepackage{algorithmic}
\ifpdf
  \DeclareGraphicsExtensions{.eps,.pdf,.png,.jpg}
\else
  \DeclareGraphicsExtensions{.eps}
\fi

\usepackage{cite,url}
\usepackage[title]{appendix}

\usepackage{multirow}
\usepackage{multicol}  

\usepackage{bbm}
\usepackage{xcolor}
\usepackage{float}

\usepackage{listings}
\usepackage{hyperref}

\newsiamremark{remark}{Remark}
\newsiamremark{hypothesis}{Hypothesis}
\crefname{hypothesis}{Hypothesis}{Hypotheses}
\newsiamthm{claim}{Claim}
\newsiamremark{fact}{Fact}
\crefname{fact}{Fact}{Facts}

\newcommand{\N}{\mathbb{N}}
\newcommand{\R}{\mathbb{R}}
\newcommand{\Rn}{\mathbb{R}^n}

\newcommand{\inner}[2]{\langle{#1},{#2}\rangle}

\newcommand{\norm}[1]{\|#1\|}

\newcommand{\J}{\mathcal{J}}

\newcommand{\ds}{\displaystyle}

\newcommand{\bi}{\begin{itemize}}
\newcommand{\ei}{\end{itemize}}
\newcommand{\ba}{\begin{array}}
\newcommand{\ea}{\end{array}}

\headers{A Cubic Regularization Method for Multiobjective Optimization}{Douglas S. Gon\c{c}alves, Max L. N. Gon\c{c}alves, Jefferson G. Melo}

\title{A Cubic Regularization Method for Multiobjective Optimization\thanks{Submitted to the editors DATE.
\funding{This work was funded by FAPESC (Fundação de Amparo à Pesquisa e Inovação do Estado de Santa Catarina) [grant number 2024TR002238].}}}

\author{Douglas S. Gon\c{c}alves\thanks{CFM, Universidade Federal de Santa Catarina, Florian\'opolis, SC 88040-900, Brazil. (\email{       douglas.goncalves@ufsc.br}). The work of this author was supported in part by CNPq (Conselho Nacional de Desenvolvimento Cient\'{i}fico e Tecnol\'ogico)  Grant 305213/2021-0.}
  \and  Max L. N. Gon\c calves
  \thanks{IME, Universidade Federal de Goi\'as, Goi\^ania, GO 74001-970, Brazil. ( \email{maxlng@ufg.br} and \email{jefferson@ufg.br}). The work of these authors was
    supported in part by CNPq (Conselho Nacional de Desenvolvimento Cient\'{i}fico e Tecnol\'ogico) Grants  405349/2021-1,  304133/2021-3 and  312223/2022-6.}
   \and Jefferson G.  Melo\footnotemark[2]}

\usepackage{amsopn}

\begin{document}

\maketitle

\begin{abstract}
This work introduces a new cubic regularization  method for solving  nonconvex unconstrained multiobjective optimization problems. At each iteration of the method, a model associated with the cubic regularization of each component of the objective function is  minimized. This model allows approximations for the first- and second-order derivatives, which must  satisfy suitable error conditions.
One interesting feature of the proposed algorithm is that  the regularization parameter of the model and the accuracy of the derivative approximations are jointly adjusted using a nonmonotone line search criterion. 
Implementations of the method, where derivative information is computed using finite difference strategies, are discussed.
From a theoretical perspective, it is shown that, under the assumption that the Hessians of the components of the multiobjective function  are globally Lipschitz continuous, the method requires at most
 $\mathcal{O}(C\epsilon^{-3/2})$ iterations to generate an $\epsilon$-approximate Pareto critical, where $C$ is related to the Lipschitz constant and the inexactness parameters. 
In particular, if the first- and second-order derivative information is computed using finite differences based solely on function values, the method requires at most \(\mathcal{O}(n^{1-\beta}\epsilon^{-3/2})\) iterations, corresponding to  \(\mathcal{O}\left(mn^{3-\beta}\varepsilon^{-\frac{3}{2}}\right)\) function evaluations, where \(n\) is the dimension of the domain of the objective function, \(m\) is the number of objectives, and \(\beta \in [0,1]\) is a  constant associated with the stepsize used in the finite-difference approximation.
We further discuss the global convergence and local convergence rate of the method. Specifically, under the local convexity assumption,  we show that the method achieves superlinear convergence when the first derivative is computed exactly, and quadratic convergence when both first- and second-order derivatives are exact.
From a practical perspective,  numerical experiments are presented to illustrate the behavior of the proposed method and compare it with a multiobjective Newton method.
\end{abstract}



\section{Introduction} \label{sec:introd}

Consider the unconstrained multiobjective optimization problem
\begin{equation} \label{vectorproblem}
\min_{x \in \R^n}    F(x),  \\
\end{equation}
where $\allowbreak F\colon \R^n \rightarrow \mathbb{R}^m$ is twice continuously differentiable. Our analysis will be conducted under the following assumptions:
for every $j\in \J:=\{1,\ldots,m\}$,
\begin{itemize}
\item[\textbf{(A1)}] the Hessian of $F_j$ is $L_j$-Lipschitz continuous, i.e., 
\begin{equation*}
\|\nabla^{2}F_j(y)-\nabla^{2}F_j(z)\|\leq L_j\|y-x\|,\quad \forall y,z\in\mathbb{R}^{n};
\end{equation*}
(For convenience, we denote $L:=\max_{j\in \J} L_j.$)
\item[\textbf{(A2)}]  there exists $F_j^*\in\mathbb{R}$ such that $F_j(y)\geq F_j^*$ for all $y\in\mathbb{R}^{n}$.
\end{itemize}

In multiobjective optimization, the concept of optimality is replaced by Pareto optimality or efficiency. A point \( x^{\ast} \in \mathbb{R}^n \) is considered  Pareto optimal or efficient if there is no other point \( x \in \mathbb{R}^n \) such that \( F(x) \leq F(x^{\ast}) \) and \( F(x) \neq F(x^{\ast}) \), where the inequality \( \leq \) is understood component-wise. Similarly, a point \( x^{\ast} \in \mathbb{R}^n \) is called weakly Pareto optimal or weakly efficient if there is no point \( x \in \mathbb{R}^n \) such that \( F(x) < F(x^{\ast}) \). 
A point \( x^{\ast} \in \mathbb{R}^n \) is locally Pareto optimal (or locally weakly Pareto optimal) if there is a neighborhood \( V \subset \mathbb{R}^n \) around \( x^{\ast} \) where \( x^{\ast} \) is Pareto optimal (or weakly Pareto optimal) for \( F \) restricted to \( V \). {A necessary condition for local Pareto optimality of \( x^{\ast} \in \mathbb{R}^n \) is given by:
\begin{equation}\label{CQ} 
{    \max_{j\in \J}\langle \nabla F_j(x^*),y-x^*\rangle \geq 0, \quad \forall y\in \R^n.}
\end{equation}
A point \( x^{\ast} \in \mathbb{R}^n \) that satisfies this condition is called Pareto critical or stationary. 
Consequently, if a point \( x \) is not Pareto critical, there exists a vector \( y \in \mathbb{R}^n \) such that $y-x$ is a descent direction for \( F \) at \( x \), meaning that there exists \( \varepsilon > 0 \) such that \( F(x + t(y-x)) < F(x) \) for any \( t \in (0, \varepsilon] \).
}

In this work,  we propose and analyze a multiobjective cubic regularization  method (M-CRM) for solving  \eqref{vectorproblem}. Basically, at each iteration of the method, a model associated with the cubic regularization of each component of the objective function is minimized. More specifically, the next iterate 
 $(x^{+}, \gamma^{+}, \lambda^{+}) \in \R^n\times \R \times \R^m_+$
is an approximate KKT point of the following smooth optimization  problem with cubic objective function and quadratic constraints
\begin{equation} \label{newtonproblem}
 \begin{array}{cl}
\ds\min_{(y,\gamma)\in\R^n\times\R}    & \;  \gamma    + \frac{\sigma}6\|y-x\|^3      \\
\mbox{s. t.} & \inner{\nabla \bar F_j (x)}{y-x} + \ds\frac{1}{2}\inner{\nabla^2 \bar F_j (x)(y-x)}{y-x}  \leq \gamma,  \quad \forall j \in \J, 
\end{array}
\end{equation}
where $x\in \R^n$,  $\sigma>0$ is a regularization parameter, and $\nabla \bar F_j(x)$ and  $\nabla^{2}\bar F_j(x)$ are  approximations  to the gradient $\nabla F_j(x)$ and the Hessian $\nabla^{2}F_j(x)$, respectively.
Note that, this smooth problem is equivalent to the non-smooth scalar-value one
 \begin{equation}\label{subprob:CubicNew-def:M-xsigma}
\min_{y\in\R^n} M_{x,\sigma}(y):=\max_{j\in \J} M_{x,\sigma}^j(y), 
\end{equation}
where 
\begin{equation}\label{def:M-xsigma-j} 
M_{x,\sigma}^j(y):=\langle \nabla \bar F_j(x),y-x\rangle+\dfrac{1}{2}\langle \nabla^{2}\bar F_j(x) (y-x),y-x\rangle+\dfrac{\sigma}{6}\|y-x\|^{3}, \;   \forall j \in \J.
\end{equation}
It is worth pointing out that the exact conditions for the derivative approximations in the M-CRM, specifically the inequalities \eqref{ineq:ApproxGrad-Hess} below, depend only on the previous information, and hence they are implementable; for further discussion on this issue in the scalar case, see, for example, \cite{11850, WANG2019146, Geovanni23}.
Implementations of the method, where information about derivatives are computed through finite difference strategies, are presented.
One interesting feature of the M-CRM is that the regularization parameter $\sigma$ and the accuracy of the derivative approximations are jointly adjusted using a nonmonotone line search criterion.

From a theoretical perspective, it is shown that the M-CRM requires at most \(\mathcal{O}(C\epsilon^{-3/2})\) iterations  to generate an \(\epsilon\)-approximate Pareto critical point, where \(C\) is related to the Lipschitz constant \(L\), the values $F^*_j,$ the initial points and the inexactness parameters. 
In particular, if the gradients are exact and the Hessians  are computed by finite differences based on gradient values,
 the M-CRM requires at most \(\mathcal{O}\left(n^{(1-2\beta)/2}\varepsilon^{-\frac{3}{2}}\right)\) iterations, and \(\mathcal{O}\left(mn^{(3-2\beta)/2}\varepsilon^{-\frac{3}{2}}\right)\)  function and gradient evaluations, to generate an \(\varepsilon\)-approximate Pareto critical point for problem~\eqref{vectorproblem} (see Definiton~\eqref{def:aprSol}), where  \(\beta \in [0,1/2]\) is a constant associated with the stepsize used in the finite-difference approximation. Additionally, 
  if the first- and second-order derivative information is computed using finite differences based solely on function values, the method requires at most \(\mathcal{O}(n^{1-\beta}\epsilon^{-3/2})\) iterations, corresponding to \ \(\mathcal{O}\left(mn^{3-\beta}\varepsilon^{-\frac{3}{2}}\right)\) function evaluations, where  \(\beta \in [0,1]\).
We further discuss the global convergence and local convergence rate of the method. Specifically, under local convexity assumption,  we show that the method achieves superlinear convergence when the first derivatives are computed exactly, and quadratic convergence when both first- and second-order derivatives are exact.
 From a practical perspective, numerical experiments on standard test sets from the literature are presented to illustrate the behavior of the approach, with comparisons to the  multiobjective Newton method with safeguards  from \cite{gonccalves2021globally}.


As is well-known in the scalar context (i.e., \( m = 1 \) in problem \eqref{vectorproblem}), the CRM is a globally convergent variant of Newton method for the unconstrained minimization of twice continuously differentiable functions \cite{GRI, NP} (see also \cite{Birgin1, Geovanni23, 11850, WANG2019146,BEL5,CGT1,10.1093,MaxSubsampled,Geovanizero,louzeiro2024} for some recent variants of the CRM). 
{ In the multiobjective setting, a \(p\)-order regularization framework for constrained multiobjective optimization was introduced in \cite{MartinezMult}. Iteration-complexity results were established under the assumption that the \(p\)-th derivatives of the objective components satisfy Hölder continuity conditions. However, it is important to note that the M-CRM does not fall within the scope of the framework proposed in \cite{MartinezMult}, as the sequence \(\{x_t\}\) generated by M-CRM does not necessarily satisfy assumptions (3) and (4) therein.
Another key distinction lies in the linesearch procedures: the method in \cite{MartinezMult} relies on a tolerance parameter \(\varepsilon > 0\), which must be chosen in advance to guarantee sufficient decrease in the objective functions. In contrast, our algorithm does not require such a parameter and does not enforce monotonicity in the functional values.
Moreover, our work establishes asymptotic convergence and provides a thorough local convergence analysis of the M-CRM, aspects that are not addressed in the framework of \cite{MartinezMult}.}
We also mention that a number of works discussing iterative methods for solving    multiobjective/vectorial problems have been recently developed. Papers dealing with this subject include, but are not limited to 
\cite{benar&fliege,bonnel2005proximal,mauriciobenar&fliege,fukuda2011convergence,fukuda2013inexact,mauricio&benar,luis-jef-yun,Glaydston1,fukudaProxGradMultiobjective,Ray1,acceProxGradMultiobjective,Leandro2022,gonccalves2021globally,cgvo,M-AFW}.

The paper is organized as follows: Section~\ref{sec:prelim}  contains  the key theoretical results that form the foundation of the proposed algorithm.
 Section~\ref{sec:MCRM} formally describes the M-CRM to solve \eqref{vectorproblem} and establishes its well-defineness. The 
  iteration-complexity and global analysis  are discussed in \ref{Sectio:iteration}, whereas   Section~\ref{sec:local} discusses local properties of the M-CRM. Approximations of gradient and Hessian of the function component $F_j$ by  using finite difference strategies are discussed in Section~\ref{sec:approximation}, and the proofs of the related results are given in the Appendix.
Section \ref{sec:numerical} contains some  numerical experiments.  
 Final remarks are given in Section~\ref{sec:final}.

\section{Preliminary material} \label{sec:prelim}

In this section, we present the key  results that form the foundation of the proposed algorithm. These results are essential for understanding both the structure of the algorithm and its theoretical analysis.

We begin by noting, from \eqref{CQ}, that \( x \) is a Pareto critical point of \eqref{vectorproblem} if and only if it is a minimum point of \( \phi(y) := \max_j \langle \nabla F_j(x), y - x \rangle \), whose the optimal value is zero. Hence, since \( \phi \) is a convex function given by the maximum of affine functions \( \phi_j(y) := \langle \nabla F_j(x), y - x \rangle \), with gradient \( \nabla F_j(x) \), it follows that \( x \) is a Pareto critical point if and only if $0$ belongs to the subdifferential of $\phi$, which is equivalent to the existence of \( \lambda \in \mathbb{R}^m_+ \) such that \( \sum_{j \in \J} \lambda_j = 1 \) and \( \sum_{j \in \J} \lambda_j \nabla F_j(x) = 0 \). 
This motivates the following concept of an approximate Pareto critical point.


\begin{definition}\label{def:aprSol}
Given a tolerance \(\varepsilon > 0\), we say that a point \(\bar{x} \in \mathbb{R}^n\) is an \(\varepsilon\)-approximate Pareto critical point of~\eqref{vectorproblem} if and only if there exists \(\bar{\lambda} \in \mathbb{R}^m_+ \) such that \( \sum_{j \in \J} \bar \lambda_j = 1 \) and \(\left\|\sum_{j \in \J} \bar{\lambda}_j \nabla F_j(\bar{x})\right\| \leq~\varepsilon\).
\end{definition}

Since problems~\eqref{newtonproblem} and \eqref{subprob:CubicNew-def:M-xsigma} are equivalent, and  the objective function of  problem~\eqref{subprob:CubicNew-def:M-xsigma} is coercive, they have a solution. 
 Additionally, problem~\eqref{newtonproblem} satisfies the Mangasarian-Fromovitz constraint qualification; for example,   $d=(0,1)$ is a strictly feasible direction.
 Thus,  there exists a multiplier $\lambda^+ \in \mathbb{R}_+^m$ such that the triple $(y,\gamma,\lambda) :=(y^+,\gamma^+,\lambda^+) \in\R^n\times \R\times\R_+^m$ satisfies the  Karush-Kuhn-Tucker conditions of \eqref{newtonproblem},  given by:
\begin{equation}\label{eq:67}
\sum_{j=1}^{m} \lambda_j=1, \quad \sum_{j=1}^{m} \lambda_j\left[\nabla \bar F_j (x) + \nabla^2 \bar F_j (x)(y-x)\right]+ \frac{\sigma}2\|y-x\|(y-x)=0,
\end{equation}
and, for all $j\in \J$,
\begin{equation*}\inner{\nabla \bar F_j (x)}{(y-x)} + \frac{1}{2}\inner{\nabla^2 \bar F_j (x)(y-x)}{y-x} \leq \gamma,
\end{equation*}
\begin{equation*}
\lambda_j\left[\inner{\nabla \bar F_j (x)}{y-x} + \frac{1}{2}\inner{\nabla^2 \bar F_j (x)(y-x)}{y-x}-\gamma\right]=0.
\end{equation*}
Note that, in view of the latter two relations,   $\gamma^+$ is uniquely determined in terms of  $(x^+,\lambda^+)$. Hence, we will consider only the pair $(x^+,\lambda^+)$ when referring to an approximate KKT solution of  problem~\eqref{newtonproblem}. Indeed,  in our algorithm,  we will consider an approximate KKT pair $(x^{+},\lambda^{+}) \in \R^n\times \R^m_+$ of problem~\eqref{newtonproblem} 
in the sense that 
\begin{equation}\label{ineq:subprobInex1}
M_{x,\sigma}(x^+)\leq 0, \quad\sum_{j=1}^{m} \lambda^+_j=1, \quad \left \|\sum_{j=1}^{m} \lambda^+_j\nabla M_{x,\sigma}^j(x^+)\right\|\leq \theta \|x^+-x\|^2,
\end{equation}
for some $\theta\geq 0$, where $M_{x,\sigma}$ is as in \eqref{subprob:CubicNew-def:M-xsigma}. 

Next, we review some well known properties of functions whose  Hessians are Lipschitz continuous. In view of  {\bf (A1)}, it can be shown that for all \( j \in \J \) and \( x, y \in \mathbb{R}^{n} \),
\begin{equation}
|F_j(y)- F_j(x)-\langle\nabla F_j(x),y-x\rangle-\dfrac{1}{2}\langle\nabla^{2}F_j(x)(y-x),y-x\rangle|\leq\dfrac{L_j}{6}\|y-x\|^{3},
\label{eq:2.2}
\end{equation}
and
\begin{equation}
\|\nabla F_j(y)-\nabla F_j(x)-\nabla^{2}F_j(x)(y-x)\|\leq\dfrac{L_j}{2}\|y-x\|^{2}.
\label{eq:2.3}
\end{equation}


We now establish a key result. 
\begin{lemma}\label{lem:basic1} 
Let $x\in \R^n$ be given and  let  $\nabla \bar F_j(x)$ and $\nabla^2 \bar F_j(x)$ be approximations of the gradient $\nabla F_j(x)$ and the Hessian $\nabla^2 F_j(x)$, respectively. Define the following residuals  
\begin{equation}\label{definition-approxGrad-Hessian}
E^G_{j}(x):=\nabla F_j(x)-\nabla\bar F_j(x),  \quad E^H_{j}(x):=\nabla^{2}F_j(x)-\nabla^{2}\bar F_j(x), \quad \forall j\in \J.
\end{equation}
 Assume that {\bf (A1)} holds and let a pair $(x^+,\sigma)\in \R^n\times\R_{++}$ be such that 
$M_{x,\sigma}(x^{+})\leq 0,$  where $M_{x,\sigma}$ is as in  \eqref{subprob:CubicNew-def:M-xsigma}.
Then, for any $\kappa_G >0$ and $j \in \J$, there holds
\begin{align}\label{ineq1:auxlem1}
F_j(x^{+})\leq  F_j(x)+ \|E^G_j(x)\|\|x^+- x\|+\frac{1}{2}\left\|E^H_j(x)\right\|\|x^{+}-x\|^2+\frac{(L-\sigma)}{6}\|x^{+}-x\|^{3}.
\end{align}
\end{lemma}
\begin{proof}
It follows from  \eqref{eq:2.2} with $y=x^+$ and the definitions of $M^j_{x,\sigma}$, and $E^G_j(x)$ and $E^H_j(x)$ in  \eqref{def:M-xsigma-j} and \eqref{definition-approxGrad-Hessian} that 
\begin{align}
F_j(x^{+})\leq& F_j(x)+\langle\nabla F_j(x),x^{+}-x\rangle+\frac{1}{2}\langle\nabla^{2}F_j(x)(x^{+}-x),x^{+}-x\rangle+\frac{L_j}{6}\|x^{+}-x\|^{3}\nonumber\\
=&  F_j(x)+ M^j_{x,\sigma}(x^{+})+ \langle E^G_j(x),x^{+}-x\rangle \nonumber\\
 &+\frac{1}{2}\left\langle  E^H_j(x)(x^{+}-x),x^{+}-x\right\rangle+\frac{L_j-\sigma}{6}\|x^{+}-x\|^{3}.\nonumber
\end{align}
Hence, using Cauchy-Schwarz inequality, property of norms operator and the definition of $M_{x,\sigma}$  in  \eqref{subprob:CubicNew-def:M-xsigma},    we have 
 \begin{align*}
F_j(x^{+})\leq &  F_j(x)+ M_{x,\sigma}(x^{+})+ \|E^G_j(x)\|\|x^{+}-x\|  +\frac{1}{2}\left\|E^H_j(x)\right\|\|x^{+}-x\|^2+\frac{L_j-\sigma}{6}\|x^{+}-x\|^{3}.
\end{align*}
Therefore, \eqref{ineq1:auxlem1} follows from the above inequality, $M_{x,\sigma}(x^{+})\leq 0$ for all $j\in \J$ and $L=\max_{j\in \J} L_j$.
\end{proof}

The next lemma provides a recursive formula for the functional value \( F_j(x^{+}) \), assuming that the errors in the gradient and Hessian of \( F_j \) at the current point—i.e., \( \|E^G_j(x)\| \) and \( \|E^H_j(x)\| \)—are controlled by quantities depending on two parameters, \(\kappa_G\) and \(\kappa_H\), and a pair \((\hat{x}, x)\), where \(\hat{x}\) represents the previous iterate in our algorithm. It is important to note that this error condition will be iteratively ensured in our algorithm through a finite difference procedure.

\begin{corollary}
\label{lem:2.2}
 Let $(x,\hat x)\in \R^n\times\R^n$ be given and assume that  $E^G_{j}(x)$ and $E^H_{j}(x)$ as  in \eqref{definition-approxGrad-Hessian} satisfy 
\begin{equation}\label{definition-approxGrad-Hessian2}
\|E^G_j(x)\|\leq\kappa_{G}\|x-\hat{x}\|^2,  \quad \|E^H_j(x)\|\leq\kappa_{H}\|x-\hat{x}\|, \quad \forall j\in \J,
\end{equation}
for some  $\kappa_{G}, \kappa_{H} \geq 0$. 
 Assume that {\bf (A1)} holds and let a pair $(x^+,\sigma)\in \R^n\times\R_{++}$ be such that 
$ 
M_{x,\sigma}(x^{+})\leq 0
$
and
\begin{equation}
\sigma\geq 2[L+3\rho(2\kappa_{G}+\kappa_{H})],
\label{ineq:threshold-sigma}
\end{equation}
for some $\rho\geq 1$. Then, the following inequality holds
\begin{equation}\label{ineq:RecursionF(x)}
F_j(x^{+})-F_j(x)\leq\dfrac{\sigma}{12\rho}\|x-\hat{x}\|^{3}-\dfrac{\sigma}{12}\|x^{+}-x\|^{3}, \quad \forall j \in \J.
\end{equation}
\end{corollary}
\begin{proof}
In view of Lemma~\ref{lem:basic1}, \eqref{definition-approxGrad-Hessian2},  and the fact that $a^2b\leq a^3+b^3$ for any $a,b\geq 0$, we have 
\begin{align*}
F_j(x^{+})- F_j(x)&\leq  \|E^G_j(x)\|\|x^+-x\|+\frac{1}{2}\|E^H_j(x)\|\|x-x^+\|^2+\frac{(L-\sigma)}{6}\|x^{+}-x\|^{3}\\
&\leq \kappa_{G}\|x-\hat{x}\|^2\|x^+-x\|+ \frac{\kappa_{H}}{2}\|x-\hat{x}\|\|x-x^+\|^2+\frac{(L-\sigma)}{6}\|x^{+}-x\|^{3}\nonumber\\
&\leq  \kappa_G\|x-\hat x\|^{3}+\frac{\kappa_H}{2}\|x-\hat x\|^3+\frac{(L+6\kappa_G+3\kappa_H-\sigma)}{6}\|x^{+}-x\|^{3}\nonumber\\
&\leq \frac{(2\kappa_G+\kappa_H)}{2}\|x-\hat x\|^{3} +\frac{[L-\sigma+3\rho(2\kappa_G+\kappa_H)]}{6}\|x^{+}-x\|^{3},
\end{align*}
where the last inequality is due to the fact that  $\rho\geq 1$.
Therefore, \eqref{ineq:RecursionF(x)} follows from the above inequality and by noting that \eqref{ineq:threshold-sigma} implies that 
\[
L-\sigma+3\rho(2\kappa_G+\kappa_H)\leq -\frac{\sigma}{2}, \qquad 
2\kappa_{G}+\kappa_{H}\leq \dfrac{\sigma}{6\rho}.
\]
\end{proof}

Our next result essentially shows that by computing an approximate KKT pair $(x^+,\lambda^+)$ of the cubic approximation subproblem as in \eqref{ineq:subprobInex1}, we can control the convex combination of the gradients of the components of the multiobjective function $\sum_{j \in \J}\lambda^+_j\nabla F_j(x^{+})$ in terms of the residuals defined in \eqref{definition-approxGrad-Hessian} and the displacement $\|x^+-x\|$.

\begin{lemma}
\label{lem3:basic-gt+1<Deltax}
Let $x\in \R^n$ be given and consider $E^G_{j}(x)$ and  $E^H_{j}(x)$ as in  \eqref{definition-approxGrad-Hessian}. Assume that {\bf (A1)} holds and let a triple $(x^{+},\lambda^{+},\sigma) \in \R^n\times \R^m_+\times\R_{++} $ be such that $\sum_{j=1}^{m} \lambda^+_j=1$ and $\left\|\sum_{j \in \J}\lambda^+_j\nabla M_{x,\sigma}^j(x^{+})\right\|\leq \theta\|x^+- x\|^2$, for some $\theta\geq 0$. Then, we have   
\begin{equation*}
\left\|\sum_{j \in \J}\lambda^+_j\nabla F_j(x^{+})\right\|\leq \frac{L+\sigma+2\theta}{2}\|x^{+}-x\|^2+\max_j\left(\|E_j^G(x)\|+ {\|E_j^H(x)\|\|x^{+}-x\|}\right). 
\end{equation*} 
\end{lemma}

\begin{proof}
Let $\Phi_{x,j}$ denote the quadratic approximation given by
\begin{equation*}
\Phi_{x,j}(y):=\langle\nabla F_j(x),y-x\rangle+\dfrac{1}{2}\langle\nabla^{2}F_j(x)(y-x),y-x\rangle, \quad \forall y\in \R^n, j \in \J.
\end{equation*}
In view of \eqref{eq:2.3} and the fact that $L=\max_{j\in \J} L_j$, we have 
\begin{equation}\label{ineq:gradF-phi}
\left\|\nabla F_j(x^{+})-\nabla\Phi_{x,j}(x^{+})\right\|\leq \frac{L}{2}\|x^+-x\|^2, \qquad \forall j\in \J.
\end{equation}
On the other hand, from the definition of $M_{x,\sigma}^j$ in \eqref{def:M-xsigma-j}, \eqref{definition-approxGrad-Hessian}, and the triangle inequality for norms, we obtain
\begin{align}\label{ineq:auxGradphi-gradMj}
\left\|\nabla\Phi_{x,j}(x^{+})-\nabla M_{x,\sigma}^j(x^{+})\right\|\leq& \left\|\nabla F_j(x)-\nabla\bar F_j(x)\right\|+\|\nabla^{2}F_j(x)-\nabla^{2}\bar F_j(x)\|\|x^{+}-x\|\nonumber\\&+\frac{\sigma}{2}\|x^{+}-x\|^2\nonumber \\
=& \|E_j^G(x)\|+\|E_j^H(x)\|\|x^+-x\|+\frac{\sigma}{2}\|x^{+}-x\|^2\nonumber.
\end{align}
It follows by using the triangle inequality for norms, the fact that $\lambda^+_j\geq 0$,  and \eqref{ineq:gradF-phi}  that
\begin{align*}
\left\|\sum_{j \in \J }\lambda^+_j\nabla F_j(x^{+})\right\|
\leq& \sum_{j \in \J}\lambda^+_j\left(\left\|\nabla F_j(x^{+})-\nabla\Phi_{x,j}(x^{+})\right\|+\left\|\nabla\Phi_{x,j}(x^{+})-\nabla M_{x,\sigma}^j(x^{+})\right\|\right)\\
&+\left\|\sum_{j \in \J}\lambda^+_j\nabla M_{x,\sigma}^j(x^{+})\right\|\\
\leq& \sum_{j \in \J}\lambda^+_j\left(\frac{L+\sigma+2\theta}{2}\|x^{+}-x\|^2+\|E_j^G(x)\|+\|E_j^H(x)\|\|x^+-x\|\right).
\end{align*} 
The proof of the lemma then follows from the above inequality and the fact that $\sum_{j \in \J}\lambda^+_j=1$.
\end{proof}

\section{Multiobjective Cubic Regularization Method}\label{sec:MCRM}

In this section, we first introduce the multiobjective cubic regularization method  for computing approximate Pareto critical points of problem \eqref{vectorproblem}. We then establish its iteration complexity and global convergence properties. The local behavior of the sheme is also studied. 
Finally, we discuss implementations of the method where the approximation of the gradient and Hessian of each function component \( F_j \) are computed  using finite difference strategies.
\\[4mm]
\textbf{Algorithm 1.} Multiobjective Cubic Regularization  Method (M-CRM)
\\[0.2cm]
\textbf{Step 0.} Choose $x_{0},x_{1}\in\mathbb{R}^{n}$, {$\sigma_{1}>0$}, $\theta\geq 0$, $\alpha>1$,  $\bar\kappa_{G}, \bar\kappa_{H}\geq 0,$  and set  $t:=1$.
\\
{\bf Step 1.}  Find the smallest integer $i\geq 0$ such that $\alpha^{i-1}\sigma_{t}\geq \sigma_{1}$.
\\
{\bf Step 1.1.} For every $j \in \J$, let  $\nabla \bar F_j(x_t)$ and $\nabla^2 \bar F_j(x_t)$ be such that $E^G_j(\cdot)$ and $E^H_j(\cdot)$, defined   in \eqref{definition-approxGrad-Hessian}, satisfy
\begin{equation}
\|E^G_j(x_t)\|\leq\frac{\bar \kappa_{G}}{\alpha^{i-1}}\|x_t-x_{t-1}\|^2,  \quad \|E^H_j(x_t)\|\leq\frac{\bar \kappa_{H}}{\alpha^{i-1}}\|x_t-x_{t-1}\|.
\label{ineq:ApproxGrad-Hess}
\end{equation}
{\bf Step 1.2.} 
Compute an approximate KKT point $(x_{t,i}^{+}, \lambda_{t,i}^{+}) \in \R^n \times \R^m_+$ of subproblem \eqref{newtonproblem} with $x:=x_t$ and $\sigma:= \alpha^i\sigma_t$ 
such that
\small
\begin{equation}
M_{x_{t},\alpha^{i}\sigma_{t}}(x^{+}_{t,i})\leq 0, \quad \sum_{j=1}^{m} (\lambda^+_{t,i})_j=1, \quad \left \|\sum_{j=1}^{m} (\lambda^+_{t,i})_j\nabla M_{x_t,\alpha^i\sigma_t}^j(x_{t,i}^+)\right\|\leq \theta  \|x_{t,i}^+-x_t\|^2,
\label{eq:subprob-Alg1}
\end{equation}
where 
\[ M_{x,\sigma}(y):=\max_{j \in \J} \,M^j_{x,\sigma}(y), \quad M^j_{x,\sigma}(y):=\langle \nabla \bar F_j(x),y-x\rangle+\dfrac{1}{2}\langle \nabla^{2}\bar F_j(x) (y-x),y-x\rangle+\dfrac{\sigma}{6}\|y-x\|^{3}.\]
\normalsize
\\
{\bf Step 1.3.} If 
\begin{equation}
F_j(x_{t})-F_j(x^{+}_{t,i})\geq\dfrac{\alpha^{i}\sigma_{t}}{12}\|x^{+}_{t,i}-x_{t}\|^{3}-\dfrac{\sigma_{t}}{12}\|x_{t}-x_{t-1}\|^{3}, \quad \forall j \in \J,
\label{ineq:linesearch}
\end{equation}
set $i_{t}=i$ and go to Step 2. Otherwise, set $i:=i+1$ and go to Step 1.1.
\\
{\bf Step 2.} Set $x_{t+1}=x^{+}_{t,i_{t}}$, $\lambda_{t+1}=\lambda^+_{t,i}$, $\sigma_{t+1}=\alpha^{i_{t}-1}\sigma_{t}$, $t:=t+1$, and go to Step 1.
\vspace{2mm}

\begin{remark} 
(i) Note that the inexact conditions in \eqref{ineq:ApproxGrad-Hess} for the derivative approximations depend solely on current information and hence are implementable; for further discussion on this issue, see, for example, \cite{11850, WANG2019146, Geovanni23}.
Implementations of the M-CRM, where information about derivatives are computed through finite difference strategies, will be presented in Subsection~\ref{sec:approximation}.
(ii) Since the first inequality in \eqref{eq:subprob-Alg1} is equivalent to 
\( M_{x_{t},\alpha^{i}\sigma_{t}}(x^{+}_{t,i}) \leq 0=M_{x_{t},\alpha^{i}\sigma_{t}}(x_{t}) \), we require that the next point \( x^+ \) provides a decrease in the model in \eqref{subprob:CubicNew-def:M-xsigma}. Additionally, the second inequality in \eqref{eq:subprob-Alg1}  represents  a relaxation of the second equality in \eqref{eq:67}. (iii) An interesting feature of the M-CRM is that the regularization parameter $\sigma_t$ and the accuracy of the derivative approximations in   \eqref{ineq:ApproxGrad-Hess} are jointly adjusted using the nonmonotone line search criterion in \eqref{ineq:linesearch}. Notably, for every $j \in \J$, 
the sequence $\{F_j(x_t)\}$ may exhibit  nonmonotonic behavior. 
\end{remark}

Next we prove that  the inner procedure in Step 1 terminates after a finite number of iterations and establish an upper bound on the sequence of regularization parameters $\{\sigma_t\}$.

\begin{lemma}
\label{lem:3.1}
Assume that {\bf(A1)} holds. Then, Algorithm~1 is well-defined and the sequence of regularization parameters $\left\{\sigma_{t}\right\}$  satisfy
\begin{equation}\label{def:sigmaMax}
\sigma_{1}\leq\sigma_{t}\leq \sigma_1+2[L+3\alpha(2\bar \kappa_{G}+\bar \kappa_{H})]:=\sigma_{max},
\end{equation}
for all $t\geq 1$. 
{Moreover, let $\delta$ be the number of evaluations of $F_j(\,\cdot\,)$ and $\nabla F_j(\,\cdot\,)$ for each inner iteration $i$. 
Then,  the total number $\delta_{T}$ of evaluations of $F_j(\,\cdot\,)$ and $\nabla F_j(\,\cdot\,)$ up to the $T$-th outer iteration is bounded as follows:
\begin{equation}
\delta_{T}\leq \delta\left[2T+\log_{\alpha}\left(2(L+3\alpha(2\bar \kappa_{G}+\bar \kappa_{H})+\sigma_1\right)-\log_{\alpha}(\sigma_{1})\right].
\label{eq:3.8}
\end{equation}}
\end{lemma}
\begin{proof}
Note first that any solution of \eqref{newtonproblem} satisfies \eqref{eq:subprob-Alg1}. To demonstrate that Algorithm~1 is well-defined, it suffices to show that, for any iteration \(t\), the number of inner iterations \(i_t\) is finite.
It is easy to see that if, at some outer iteration \(t\), Algorithm 1 executes exactly \(i\) inner iterations such that \(\alpha^i\sigma_t \leq \sigma_{\text{max}}\), then, since \(\alpha > 1\), \(\sigma_t\) satisfies the second inequality in \eqref{def:sigmaMax}. Considering the definitions of \(i_t\) and \(\sigma_{t+1}\) in Step 2, it follows that \(i_t = i\) and \(\sigma_{t+1} = \alpha^{i-1}\sigma_t \leq \alpha^i\sigma_t \leq \sigma_{\text{max}}\). This shows that \(\sigma_{t+1}\) also satisfies the second inequality in \eqref{def:sigmaMax}.
Now, assume that Algorithm 1 executes a number of inner iterations \(i\) such that \(\alpha^i\sigma_t > \sigma_{\text{max}}\), and let \(i\) be the first inner iteration where this inequality holds. In this case, we have 
\[
\alpha^{i}\sigma_t > \sigma_{\text{max}} \geq 2\left(L + 3\alpha\left(2\bar{\kappa}_{G} + \bar{\kappa}_{H}\right)\right) = 2\left(L + 3\alpha^{i}\left(\frac{2\bar{\kappa}_{G}}{\alpha^{i-1}} + \frac{\bar{\kappa}_{H}}{\alpha^{i-1}}\right)\right).
\]
Hence, in view of \eqref{ineq:ApproxGrad-Hess} and the first inequality in \eqref{eq:subprob-Alg1}, the assumptions of Corollary~\ref{lem:2.2} hold with \(\sigma := \alpha^{i}\sigma_t\), \(x := x_t\), \(x^+ := x^{+}_{t,i}\), \(\hat{x} := x_{t-1}\), \(\rho := \alpha^i\), \(\kappa_G := \bar{\kappa}_G/\alpha^{i-1}\), and \(\kappa_H := \bar{\kappa}_H/\alpha^{i-1}\). Therefore, we have 
\[
F_j(x_t) - F_j(x^{+}_{t,i}) \geq \frac{\alpha^i\sigma_t}{12} \|x^{+}_{t,i} - x_t\|^3 - \frac{\alpha^i\sigma_t}{12\alpha^i} \|x_t - x_{t-1}\|^3 \\
= \frac{\alpha^i\sigma_t}{12} \|x^{+}_{t,i} - x_t\|^3 - \frac{\sigma_t}{12} \|x_t - x_{t-1}\|^3, \quad \forall j \in \mathcal{J},
\]
which implies that the inequality in \eqref{ineq:linesearch} holds at the \(i\)-th inner iteration. Therefore, by the definition of \(i_t\) and \(\sigma_{t+1}\) in Step 2, we must have \(i_t = i\) and \(\sigma_{t+1} = \alpha^{i-1}\sigma_t\). Moreover, since \(i\) is the first inner iteration such that \(\alpha^i\sigma_t > \sigma_{\text{max}}\), we conclude that \(\alpha^{i-1}\sigma_t \leq \sigma_{\text{max}}\), implying, in particular, that \(\sigma_t \leq \sigma_{\text{max}}\), due to the fact that \(\alpha > 1\).
Therefore, Algorithm 1 is well-defined, and the second inequality in \eqref{def:sigmaMax} holds for any outer iteration \(t\). Finally, note that \(\sigma_1\), given in Step 0 of Algorithm 1, trivially satisfies \eqref{def:sigmaMax}, and in view of Step 1 and the definition of \(\sigma_{t+1}\) in Step 2, the first inequality in \eqref{def:sigmaMax} also trivially holds. Thus, the proof of the first part of  the lemma is concluded.

In order to prove the last statement of the lemma, note first that 
\begin{equation*}
\sigma_{t+1}=\alpha^{i_{t}-1}\sigma_{t}\Longrightarrow (i_{t}+1)=2+\log_{\alpha}(\sigma_{t+1})-\log_{\alpha}(\sigma_{t}).
\end{equation*}
Thus,
\begin{equation*}
\delta_{T}\leq\sum_{t=1}^{T}\delta(i_{t}+1)=\delta[2T+\log_{\alpha}(\sigma_{T+1})-\log_{\alpha}(\sigma_{1})],
\end{equation*}
which, combined with \eqref{def:sigmaMax}, yields \eqref{eq:3.8}.\end{proof}

\subsection{Iteration-complexity for Algorithm~1}\label{Sectio:iteration}

In this subsection, we establish an iteration-complexity bound on the number of  outer iterations of Algorithm~1 required to generate an \(\epsilon\)-approximate Pareto critical point of \eqref{vectorproblem}.

The following quantities will be useful in our analysis:
{\small
\begin{equation}\label{def:delta0,1}
\displaystyle \Delta_0:=\left(\dfrac{L+2\theta+(\sigma_{max}+2\bar\kappa_G+2\bar \kappa_H)\alpha}{2}\right)^{3/2}, \Delta_1:=\frac{24\min_{j\in \J}[F_j(x_{1})-F_j^*]}{(\alpha-1)\sigma_{1}}+\frac{\alpha+1}{\alpha-1}\|x_{1}-x_{0}\|^{3},
\end{equation}}
where $F_j^*$ is as in {\bf (A2)} and $\sigma_{max}$ is as in \eqref{def:sigmaMax}.

\begin{theorem}\label{thm:sumNablaF_j<-Delta0}
Assume that {\bf (A1)} and {\bf (A2)} hold. Then, for every $T\geq 2$, the following inequalities hold
\begin{equation}\label{ineq:boundTheorem}
 \sum_{t=1}^{T-1}\|x_{t}-x_{t-1}\|^{3}\leq\Delta_1, \qquad \sum_{t=2}^{T}\left\| \sum_{j\in \J}(\lambda_{{t}})_j\nabla F_j(x_{t})\right\|^{3/2}\leq \Delta_0\times\Delta_1.  
\end{equation}
As a consequence,  given $\varepsilon>0$, Algorithm~1 needs at most {$\mathcal{O}\left(\Delta_0\times \Delta_1\varepsilon^{-\frac{3}{2}}\right)$} iterations to generate  an  $\varepsilon-$approximate Pareto critical point for problem~\eqref{vectorproblem} in the sense of Definition~\ref{def:aprSol}.
\end{theorem}
\begin{proof}
Since $\alpha^{i_{t}-1}\sigma_{t}=\sigma_{t+1}$,  it follows from \eqref{ineq:linesearch} that
\begin{equation*}
F_j(x_{t})-F_j(x_{t+1})\geq\dfrac{\alpha\sigma_{t+1}}{12}\|x_{t+1}-x_{t}\|^{3}-\dfrac{\sigma_{t}}{12}\|x_{t}-x_{t-1}\|^{3},\quad \forall j\in \J, \;\mbox{and}\;  t=1,\ldots,T-1.
\end{equation*}
Summing up the above inequality from $t=1$ to $t=T-1$ and  using the facts $F_j(x_t)\geq F_j^*$ and  $\sigma_{t}\geq\sigma_{1}$, we obtain
\begin{align}\nonumber
F_j(x_{1})-F_{j}^*&\geq  F_j(x_{1})-F_j(x_{T}) =  \sum_{t=1}^{T-1}\left[F_j(x_{t})-F_j(x_{t+1})\right]\\\nonumber
                &\geq \sum_{t=1}^{T-1}\dfrac{(\alpha-1)\sigma_{t+1}}{12}\|x_{t+1}-x_{t}\|^{3}+\dfrac{1}{12}\sum_{t=1}^{T-1}\left[\sigma_{t+1}\|x_{t+1}-x_{t}\|^{3}-\sigma_{t}\|x_{t}-x_{t-1}\|^{3}\right]\\
                &\geq \sum_{t=1}^{T-1}\dfrac{(\alpha-1)\sigma_{1}}{12}\|x_{t+1}-x_{t}\|^{3}+ \dfrac{\sigma_{T}}{12}\|x_{T}-x_{T-1}\|^{3}-\dfrac{\sigma_{1}}{12}\|x_{1}-x_{0}\|^{3}, \label{er:lo90}
\end{align}
which implies that 
\begin{equation*}
\sum_{t=1}^{T-1}\|x_{t+1}-x_{t}\|^{3}\leq\dfrac{12\min_{j\in \J}[F_j(x_{1})-F_j^*]+\sigma_{1}\|x_{1}-x_{0}\|^{3}}{(\alpha-1)\sigma_{1}}=:\tilde\Delta_1.
\end{equation*}
Hence, for every $t=1,\ldots,T-1$, we have 
\begin{equation}\label{ineq:sumDeltaj}
\sum_{t=1}^{T-1}\max\{\|x_{t+1}-x_{t}\|,\|x_{t}-x_{t-1}\|\}^3\leq \sum_{t=1}^{T-1}\left[\|x_{t+1}-x_{t}\|^3+\|x_{t}-x_{t-1}\|^3\right]\leq 2\tilde\Delta_1+\|x_1-x_0\|^3=\Delta_1.
\end{equation}
The first inequality in \eqref{ineq:boundTheorem} follows from the above inequality. 
Now, it follows from Lemma~\ref{lem3:basic-gt+1<Deltax} with $(x,,x^+,\lambda^+,\sigma):=(x_t,x_{t+1},\lambda_{t+1},\alpha^{i_t}\sigma_t)$ and from \eqref{ineq:ApproxGrad-Hess} and \eqref{eq:subprob-Alg1} that 
 \begin{align*}
\left\|\sum_{j\in \J}(\lambda_{t+1})_j\nabla F_j(x_{t+1})\right\| & \leq \frac{L+\alpha^{i_t}\sigma_t+2\theta}{2}\|x_{t+1}-x_t\|^2 \\
\ &+\frac{\bar\kappa_{G}}{\alpha^{i_t-1}}\|x_t-x_{t-1}\|^2+ \frac{\bar\kappa_{H}}{\alpha^{i_t-1}}\|x_t-x_{t-1}\|\|x_{t+1}-x_t\|.
\end{align*} 
Hence, since $\alpha^{i_t}\geq 1$ and $\sigma_{t+1}=\alpha^{i_t-1}\sigma_t$, it follows from the above inequality, the second inequality in \eqref{def:sigmaMax} for $\sigma_{t+1}$ and by using the inequality $ab\leq (a^2+b^2)/2$ with $a=\|x-\hat x\|$ and $b=\|x^+-x\|$, that 
\begin{align*}
\left\|\sum_{j \in \J}(\lambda_{t+1})_j\nabla F_j(x_{t+1})\right\|&\leq \frac{L+\alpha\sigma_{max}+2\theta+\alpha\bar \kappa_H}{2}\|x_{t+1}-x_t\|^2+\frac{(2 \bar \kappa_{G}+\bar \kappa_H)\alpha}{2}\|x_t-x_{t-1}\|^2\\
&\leq \Delta_0^{2/3}\max\{\|x_{t+1}-x_{t}\|,\|x_{t}-x_{t-1}\|\}^2,
\end{align*} 
where the last inequality is due to the definition of $\Delta_0$ in \eqref{def:delta0,1}. Hence,  we have
$$
\left\|\sum_{j\in \J }(\lambda_{t+1})_j\nabla F_j(x_{t+1})\right\|^{3/2}\leq \Delta_0\max\{\|x_{t+1}-x_{t}\|,\|x_{t}-x_{t-1}\|\}^3.
$$
Therefore, the second inequality in \eqref{ineq:boundTheorem} follows by summing both sides of the above inequality from $t=1$ to $t=T-1$ and by using \eqref{ineq:sumDeltaj}. 

Now note that  \eqref{ineq:boundTheorem} implies that, for every $T\geq 2$, the following inequality holds
\begin{equation}\label{ineq:aux-complexity}
\min_{t\in \{2\ldots T\}}\left\| \sum_{j\in \J }(\lambda_{{t}})_j\nabla F_j(x_{t})\right\|\leq \left(\frac{\Delta_0\times\Delta_1}{T-1}\right)^{2/3}.
\end{equation}
The last statement of the theorem follows immediately from the above inequality with  $T=1+\lceil\Delta_0\times \Delta_1\varepsilon^{-3/2}\rceil$ and 
$(\bar x, \bar \lambda):=(x_{\bar t},\lambda_{\bar t})$, where 
 $\bar t$ is the minimizer of the left hand side of \eqref{ineq:aux-complexity}.
\end{proof}

 The following convergence result follows trivially   from \eqref{ineq:boundTheorem}. In particular, it shows that any limit point of the sequence $\{x_t\}$ generated by Algorithm~1 is a Pareto critical of \eqref{vectorproblem}. Note that, since for any outer iteration $t$,  $\lambda_t\in\R^m_+$ and $\sum_{j \in \J}(\lambda_{t})_j=1$, we have that
 $\{\lambda_t\}$ is bounded, and hence it has accumulation point.
\begin{corollary}\label{cor:conv}
Assume that {\bf (A1)} and {\bf (A2)} hold and let $\{x_{t}\}$ be generated by Algorithm~1. Then, 
\begin{equation}
\lim_{t\to+\infty} \|x_t-x_{t-1}\|=0, \qquad \lim_{t\to+\infty}\left \|\sum_{j\in\J}(\lambda_{t})_j\nabla F_j(x_{t})\right\|=0.
\label{eq:3.16}
\end{equation}
As a consequence, all limit points of $\{x_t\}$, if any, are weak Pareto critical for \eqref{vectorproblem}.
\end{corollary}

\begin{remark}
It follows from Corollary~\ref{cor:conv} that if \( m = 1 \), then \(\{\nabla F(x_t)\}_{t \geq 1}\) converges to zero. Consequently, any limit point of \(\{x_t\}\) is critical for problem~\eqref{vectorproblem}. This result strengthens the one in \cite[Corollary~2]{Geovanni23}, where it was established that, in the scalar case, \(\liminf_t \|\nabla F(x_t)\| = 0\).
\end{remark}

\subsection{Local  convergence rates}\label{sec:local}

In this subsection, we establish local convergence rates for the sequence \(\{x_t\}\) generated by Algorithm~1 under some additional assumptions. One of the key assumptions is that \( F_j \) is strongly convex on a specific level set, as stated next.
\vspace{2mm}

\noindent \textbf{(A3)} There exists $\mu >0$ and an index $t_0$ such that, for every $j\in \J$, $\nabla^{2}F_j(x)\succeq\mu I$ whenever
\begin{equation*}
F_j(x)\leq F_j(x_{t_0})+\dfrac{\sigma_{1}}{12}\|x_1-x_0\|^3,
\end{equation*}
where $\sigma_1$ is as in Step~0 of Algorithm~1.
\vspace{2mm}

First note that, by similar arguments as in  \eqref{er:lo90}, we have,
 for all $t>t_0$,
\begin{align*}
F_j(x_{t})-F_j(x_{t_0})&=\sum_{i=t_0}^{t-1}F_j(x_{i+1})-F_j(x_{i})\\
&\leq  \frac{\sigma_{1}}{12}\|x_{1}-x_{0}\|^{3} -\sum_{j=t_0}^{t-1}\dfrac{(\alpha-1)\sigma_{1}}{12}\|x_{j+1}-x_{j}\|^{3}- \dfrac{\sigma_{t}}{12}\|x_{t}-x_{t-1}\|^{3}\\
&\leq \frac{\sigma_{1}}{12}\|x_{1}-x_{0}\|^{3}.
\end{align*} 
Thus, it follows from {\bf (A3)} that 
\begin{equation}
\nabla^{2}F_j(x_{t})\succeq\mu I,\quad\forall t\geq t_0.
\label{ineq:nabla2F defpositiva}
\end{equation}

The local convergence rate will  require  the gradient \(\nabla F_j\) to be computed exactly and imposes a slightly stronger condition on the approximations of Hessian \(\nabla^2 F_j\). For clarity, these conditions are presented below as an assumption.
\\[2mm]
{\bf (A4) } For every \( j \in \J \), the gradient of $F_j$ is exactly computed  (i.e., \(\nabla \bar{F}_j = \nabla F_j\)), and the Hessian approximation \(\nabla^2 \bar{F}_j\) satisfies the following condition:
\begin{equation}
 \|E^H_j(x_t)\|\leq\frac{\bar \kappa_{H}}{\alpha^{i-1}}\min\left\{ \|\Delta x_{t}\|, \zeta\left\|g_t\right\| \right\}.
\label{ineq:ApproxGrad-Hess12}
\end{equation}
where  $E^H_j$ is as in \eqref{definition-approxGrad-Hessian}, $\bar\kappa_H$ is as in Step~0 of Algorithm~1, $\zeta>0$ and
\begin{equation}\label{def:gt-Deltaxt}
 g_t:=   \sum_{j\in \J} (\lambda_{t})_j\nabla F_j(x_{t}), \qquad \Delta x_t:=x_t-x_{t-1}.
\end{equation}

The next result establishes key inequalities that will be used to demonstrate the superlinear or quadratic convergence of \(\{x_t\}\). In particular, the sequence \(\{g_t\}\) converges quadratically to zero, and the sequence \(\{x_t\}\) converges to a local Pareto solution of \eqref{vectorproblem}.


\begin{theorem} \label{thm:3.3}
Assume that  {\bf (A1)-(A4)} hold and let $\left\{x_{t}\right\}$ be generated by Algorithm 1. Then, there exists $t_1\in \N$ such that, for all $\lambda \in \R^m_+$ such that $\sum_j\lambda_j=1$, there hold
\begin{equation}\label{eq:Deltaxt-lambdanablaFxt} 
\|\Delta x_{t+1}\|\leq \frac{4}{\mu}\left\|\sum_{j\in \J}\lambda_j \nabla F_j(x_t)\right\|, \quad  \forall t\geq t_1,
\end{equation}
\begin{equation}
\left \|g_{t+1}\right\|\leq C\left\|g_t\right\|^{2},  \quad \forall t\geq t_1,
\label{eq:nablaFxt+1<=nablaFxt}
\end{equation}
where $C:= 8\left(L+\alpha\sigma_{max}+2\theta\right)/\mu^2+4\alpha\bar\kappa_H\zeta/\mu$.
Furthermore, the whole sequence $\{x_t\}$ converges to a local Pareto solution of \eqref{vectorproblem}.
\end{theorem}
\begin{proof} 
It follows from Corollary~\ref{cor:conv} that $\{g_t\}_{t\geq 1}$ converges to zero and hence there exists $t_1\geq t_0$ such that 
\begin{equation}
\left\|g_t\right\|=\left\|\sum_{j\in \J}(\lambda_t)_j\nabla F_j(x_{t})\right\|\leq\min\left\{\dfrac{\mu}{2\alpha\bar \kappa_{H}\zeta},\frac{1}{2C}\right\}, \quad \forall  t\geq t_1.
\label{eq:3.28}
\end{equation}
 Hence, by \eqref{ineq:ApproxGrad-Hess12} and  the fact that $\alpha^{i}\geq 1$,  we obtain, for every $v\neq 0$, that 
\begin{equation*}
v^T(\nabla^{2}F_j(x_{t})-\nabla^2\bar F_j(x_t))v=v^TE^H_j(x_t)v\leq  \|E^H_j(x_t)\|\|v\|^2\leq\alpha\bar \kappa_{H}\zeta\left\|g_t\|\right\|\|v\|^2\leq\dfrac{\mu}{2}\|v\|^2, 
\end{equation*}
which, combined with (\ref{ineq:nabla2F defpositiva}), yields 
\begin{equation}\label{eq:3.30}
 \nabla^2\bar F_j(x_t)\succeq \nabla^{2}F_j(x_{t})-\dfrac{\mu}{2}I\succeq\dfrac{\mu}{2}I.
\end{equation}
On the other hand, it follows from the first inequality in (\ref{eq:subprob-Alg1}) and by using $\Delta x_{t+1}=x_{t+1}-x_t$ that 
\begin{equation*} M_{x_t,\alpha^{i_t}\sigma_t}^j(x_{t+1})=\langle \nabla F_j(x_t),\Delta x_{t+1}\rangle+\dfrac{1}{2}\langle \nabla^{2}\bar F_j(x_t) \Delta x_{t+1},\Delta x_{t+1}\rangle+\dfrac{\alpha^{i_t}\sigma}{6}\|\Delta x_{t+1}\|^{3}\leq 0,
\end{equation*}
which implies that
\begin{equation*} 
\dfrac{1}{2}\langle \nabla^{2}\bar F_j(x_t) \Delta x_{t+1},\Delta x_{t+1}\rangle\leq -\langle \nabla  F_j(x_t),\Delta x_{t+1}\rangle.
\end{equation*}
From the last inequality and \eqref{eq:3.30}, we obtain 
\begin{equation*} \frac{\mu}{4}\|\Delta x_{t+1}\|^2\leq \dfrac{1}{2}\langle \nabla^{2}\bar F_j(x_t) \Delta x_{t+1},\Delta x_{t+1}\rangle\leq -\langle \nabla F_j(x_t),\Delta x_{t+1}\rangle.
\end{equation*}
Hence, for any $\lambda\in \R^m_{+}$ such that $\sum_{j\in \J}\lambda_j=1$, it follows from the last inequality and the Cauchy-Schwarz inequality that 
\begin{equation*} \frac{\mu}{4}\|\Delta x_{t+1}\|^2\leq -\left\langle\sum_{j\in \J}\lambda_j \nabla  F_j(x_t),\Delta x_{t+1}\right\rangle\leq \left\|\sum_{j\in \J}\lambda_j \nabla  F_j(x_t)\right\|\|\Delta x_{t+1}\|.
\end{equation*}
Thus,  we obtain 
\begin{equation*}
\|\Delta x_{t+1}\|\leq \frac{4}{\mu}\left\|\sum_{j\in \J}\lambda_j \nabla  F_j(x_t)\right\|,
\end{equation*}
which proves \eqref{eq:Deltaxt-lambdanablaFxt}. Now,  it follows from Lemma~\ref{lem3:basic-gt+1<Deltax} with $(x,,x^+,\lambda^+,\sigma)=(x_t,x_{t+1},\lambda_{t+1},\alpha^{i_t}\sigma_t)$, assumption {\bf (A4)}, and the definitions of $g_t,g_{t+1}$, and $\Delta x_{t+1}$ given in \eqref{def:gt-Deltaxt}, that 
 \begin{align*}
\left\|g_{t+1}\right\|&\leq \frac{L+\alpha^{i_t}\sigma_t+2\theta}{2}\|\Delta x_{t+1}\|^2+ \|\Delta x_{t+1}\|\frac{\bar \kappa_{H}}{\alpha^{i_t-1}}\min\left\{ \|\Delta x_{t}\|, \zeta\left\|g_t\right\| \right\}\\
&\leq \frac{8(L+\alpha\sigma_{max}+2\theta)}{\mu^2}\|g_t\|^2+ \frac{4\alpha\bar \kappa_{H}}{\mu} \zeta\left\|g_t\right\|^2,
\end{align*} 
where the last inequality is due to  $\alpha^{i_t}\geq 1$,  $\sigma_{t+1}=\alpha^{i_t-1}\sigma_t$,  the second inequality in \eqref{def:sigmaMax} with $t+1$, and \eqref{eq:Deltaxt-lambdanablaFxt} with $\lambda=\lambda_t$. The above inequality proves \eqref{eq:nablaFxt+1<=nablaFxt}, in view of the definition of $C$.

Now, it follows from \eqref{eq:nablaFxt+1<=nablaFxt} combined with \eqref{eq:3.28}  that 
\begin{equation} \label{ineq:contraction gt}
\|g_{t+1}\|\leq C\|g_t\|^2\leq (1/2)\|g_t\|, \quad \forall t\geq t_1.
\end{equation}
As a consequence, we have $\|g_t\|\leq 2^{-(t-t_1)}\|g_{t_1}\|$ for any $t\geq t_1$, which in turn implies that $\sum_t \|g_t\| <\infty$. Thus, in view of \eqref{eq:Deltaxt-lambdanablaFxt} with $\lambda=\lambda_t$, we also conclude that $\sum_t\|\Delta x_{t}\|<\infty$. The latter inequality immediately implies that $\{x_t\}$ is a Cauchy sequence and hence convergent to a point $x^*$. From the definition of $g_t$ and the facts that $\lambda_t\in \R^m_+$, $\sum_{j\in \J}(\lambda_t)_j=1$ for any $t\geq 1$  and $\{g_t\}_{t\geq 1}$ converges to zero (see Corollary~\ref{cor:conv}), we conclude that $\{\lambda_t\}$ has an accumulation point $\lambda^*\in \R^m_{+}$ such that $\sum_{j\in \J}\lambda^*_j=1$ and $\sum_{j\in \J} \lambda^*_j\nabla  F_j(x^*)=0$. Moreover, \eqref{ineq:nabla2F defpositiva} implies that $\sum_{j\in \J} \lambda^*_j\nabla^2  F_j(x^*)\succeq\mu I$  and hence $x^*$ is a local strict  minimum point of the function $q(x):=\sum_{j\in \J} \lambda^*_j F_j(x)$. Therefore, since $\lambda^*\in \R^m_{+}$ and $\sum_{j\in \J}\lambda^*_j=1$, we conclude that   
$$\max_{j\in \J}\{F_j(x)-F_j(x^*)\}\geq \sum_{j\in \J} \lambda^*_j (F_j(x)-F_j(x^*))=q(x)-q(x^*)>0,$$
for any $x$ in a neighborhood of  $x^*$. Thus, we conclude  that $x^*$ is a local Pareto solution of \eqref{vectorproblem}.
\end{proof}

The next  result establishes, in particular, the superlinear convergence of the sequence \(\{x_t\}\).

\begin{theorem}\label{theorem:superlinearconvergence}
Assume that {\bf (A1)--(A4)} hold and let $x^*$ be the limit point of $\{x_t\}$. Then,   there exist $C_1,C_2, C_3>0$ such that, for all $t\geq t_1$, there hold 
\begin{equation}\label{ineq:xi-xt<gt+1}
\|x_{t+1}-x^*\|\leq C_1\|g_{t+1}\| ,  
\end{equation} 
\begin{equation}\label{aux-eq: tilde gt}
  \left\|\tilde g_t\right\| \leq C_2 \min\{\|\Delta x_{t+1}\|, \|x_t-x^*\|\},
 \end{equation}
 \begin{equation}\label{ineq:xt+1-xbar leq gt xt-xbar}
  \left\|x_{t+1}-x^*\right\| \leq C_3 \|g_t\|\|x_t-x^*\|,
 \end{equation}
where $\tilde g_t:=\sum_{j\in \J} (\lambda_{t+1})_j\nabla F_j(x_t)$. As a consequence,  $\{x_t\}$ is superlinearly convergent  to $x^*$.
\end{theorem}
\begin{proof}
It follows from \eqref{eq:Deltaxt-lambdanablaFxt}  with $\lambda=\lambda_t$ and \eqref{ineq:contraction gt} that, for all $s\geq t+1$,
\begin{align*}
\|x_s-x_{t+1}\|&\leq \sum_{k=t+1}^s\|x_{k+1}-x_{k}\|=\sum_{k=t+1}^s\|\Delta x_{k+1}\|\nonumber\\
&\leq \frac{4}{\mu}\sum_{k=t+1}^s\|g_k\|=\frac{4}{\mu}\left[\|g_{t+1}\|+\|g_{t+2}\|+\cdots+ \|g_{s}\|\right]\nonumber\\
&\leq \frac{4}{\mu}\left[\|g_{t+1}\|+\frac{1}{2}\|g_{t+1}\|+\frac{1}{2^2}\|g_{t+1}\|+\cdots+ \frac{1}{2^{s-t-1}}\|g_{t+1}\|\right]\nonumber\\
&=\frac{4}{\mu}\left\{\sum_{l=0}^{s-t-1}\frac{1}{2^l}\|g_{t+1}\|\right\}\leq \frac{8}{\mu}\|g_{t+1}\|.
\end{align*}
Since  $\{x_t\}$ converges to $x^*$ (see Theorem~\ref{thm:3.3}), we conclude that \eqref{ineq:xi-xt<gt+1} holds with $C_1:=8/\mu.$

Now, in view of the last inequality in \eqref{eq:subprob-Alg1} and the definition of $\sigma_t$ in Step~2 of Algorithm~1, we have  $\|\sum_{j}(\lambda_{t+1})_j\nabla M^j_{x_t,\alpha\sigma_{t+1}}(x_{t+1})\|\leq \theta\|\Delta x_{t+1}\|^2$, which, combined with   the definitions of $M^j_{x_t,\alpha\sigma_{t+1}}(\cdot)$ and $\tilde g_t$, yields
\begin{equation}\label{ineq:aux-tilde gt-theta}
\left\|\tilde g_t +\sum_j(\lambda_{t+1})_j\nabla^2 \bar F_j(x_t)\Delta x_{t+1}+\alpha\sigma_{t+1}\|\Delta x_{t+1}\|\Delta x_{t+1}\right\|\leq \theta\|\Delta x_{t+1}\|^2.
\end{equation}
Note that in view of the definition of $E^H_j(\cdot)$ in \eqref{ineq:ApproxGrad-Hess}, we have
$$
\|\nabla^2 \bar  F_j(x_t)\|\leq \|\nabla^2 F_j(x_t)\|+\|E^H_j(x_t)\|,
$$
which, combined with  \eqref{ineq:ApproxGrad-Hess12} and the boundedness of $\{x_t\}$, implies that  $\{\nabla^2\bar F_j(x_t)\}$ is also bounded for any $j \in \J$. Hence, there exists $\Lambda>0$ such that $\|\nabla^2\bar F_j(x_t)\|\leq \Lambda$ for any $t\geq 0$ and $j \in \J$. Thus, \eqref{ineq:aux-tilde gt-theta} combined with the triangle inequality and the facts that $\sigma_t\leq \sigma_{max}$ and $\sum_j(\lambda_t)_j=1$, imply that
 \begin{align*}
\left\|\tilde g_t\right\| &\leq   \left\|\sum_j(\lambda_{t+1})_j\nabla^2 \bar F_j(x_t)\Delta x_{t+1}\right\|+\alpha\sigma_{t+1}\|\Delta x_{t+1}\|^2+\theta\|\Delta x_{t+1}\|^2 \\
&\leq \left(\Lambda+(\alpha\sigma_{max}+\theta)\|\Delta x_{t+1}\|\right)\|\Delta x_{t+1}\|.
 \end{align*}
 Since $\{\|\Delta x_{t}\|\}$ converges to zero, it is bounded, and then the above inequality implies that  there exists $\tilde C_2>0$ such that 
 \begin{equation}\label{aux-ineq123}
 \|\tilde g_t\|\leq \tilde C_2\|\Delta x_{t+1}\|,\qquad \forall t\geq 1.
 \end{equation}
Now,  it follows from Lemma~\ref{lem3:basic-gt+1<Deltax} with $(x,,x^+,\lambda^+,\sigma)=(x_t,x_{t+1},\lambda_{t+1},\alpha^{i_t}\sigma_t)$, assumption {\bf (A4)}, and the definitions of $g_t, \tilde g_t, g_{t+1}$, and $\Delta x_{t+1}$ given in \eqref{def:gt-Deltaxt}, that 
 \begin{align*}
\left\|g_{t+1}\right\|&\leq \frac{L+\alpha^{i_t}\sigma_t+2\theta}{2}\|\Delta x_{t+1}\|^2+ \|\Delta x_{t+1}\|\frac{\bar \kappa_{H}}{\alpha^{i_t-1}}\min\left\{ \|\Delta x_{t}\|, \zeta\left\|g_t\right\| \right\}\\
&\leq \frac{8(L+\alpha\sigma_{max}+2\theta)}{\mu^2}\|g_t\|\|\tilde g_t\|+ \frac{4\alpha\bar \kappa_{H}}{\mu} \zeta\left\|g_t\right\|\|\tilde g_t\|,
\end{align*} 
where the last inequality is due to  $\alpha^{i_t}\geq 1$,  $\sigma_{t+1}=\alpha^{i_t-1}\sigma_t$,  the second inequality in \eqref{def:sigmaMax} with $t+1$, and \eqref{eq:Deltaxt-lambdanablaFxt} twice, first with $\lambda=\lambda_{t+1}$ and second with $\lambda=\lambda_t$.
Hence, defining $\tilde C_1:=8(L+\alpha \sigma_{max}+2\theta)/\mu^2+4\alpha\bar\kappa_H\zeta/\mu$, we conclude that 
\begin{equation}\label{ineq:aux-gt+1<gtDelta xt}
\|g_{t+1}\|\leq \tilde C_1\|\tilde g_t\|\|g_t\|.
\end{equation}

Hence, \eqref{ineq:xi-xt<gt+1},  \eqref{aux-ineq123}, \eqref{ineq:aux-gt+1<gtDelta xt}, the fact that $\Delta x_{t+1}=x_{t+1}-x_t$, and the Cauchy-Schwarz inequality imply that 

\begin{align}
 \|x_t-x^*\|&\geq \|x_{t+1}-x_t\|-\|x_{t+1}- x^*\| \geq \frac{1}{\tilde C_2} \left\|\tilde g_t\right\| - C_1 \left\|g_{t+1}\right\| \\
 &= \left(\frac{1}{\tilde C_2}  -C_1\tilde C_1 \left\| g_t\right\|\right)\left\|\tilde g_t\right\|
\end{align}
Since, in view of Corollary~\ref{cor:conv}, $\{g_t\}$ converges to zero,  the above inequality implies that there exists $\hat C_2>0$ such that for any $t$ sufficiently large,
$$
\|x_t- x^*\|\geq \hat C_2\left\|\tilde g_t\right\|.$$
Hence, the above inequality combined with \eqref{aux-ineq123} imply that  \eqref{aux-eq: tilde gt} holds with  $C_2=\max\{\tilde C_2, 1/\hat C_2\}$.
Therefore, in view of \eqref{ineq:xi-xt<gt+1}, \eqref{aux-eq: tilde gt}, and \eqref{ineq:aux-gt+1<gtDelta xt}, we conclude that for all $t$ sufficiently large, we have
$$
\|x_{t+1}-x^*\|\leq C_1\|g_{t+1}\|\leq C_1\tilde C_1\left\|\tilde g_t\right\|\|g_t\|\leq {C_1\tilde C_1C_2}\|g_t\|\|x_t-x^*\|,
$$
which proves \eqref{ineq:xt+1-xbar leq gt xt-xbar} by defining $C_3= C_1\tilde C_1 C_2$.
The proof of the last statement of the theorem follows immediately from \eqref{ineq:xt+1-xbar leq gt xt-xbar} and the fact that $\{g_t\}$ converges to zero, in view of Corollary~\ref{cor:conv}.
\end{proof}

The next result shows that if the gradient and Hessian of $F_j$ are exactly computed  in  Algorithm~1, then the sequence $\{x_t\}$ is quadratically convergent to a local Pareto solution of \eqref{vectorproblem}.

\begin{theorem}\label{Theorem:quadraticConvergence} 
Let $\{x_t\}$ be generated by Algorithm~1 and assume that, for every $j\in \J$ and $t\geq 1$, $\nabla \bar F_j(x_t)= \nabla F_j(x_t)$ 
and $\nabla^2 \bar F_j(x_t)=\nabla^2  F_j(x_t)$.
 Moreover, assume that {\bf (A1)--(A3)} hold.
 Then,   there exists $D>0$ such that for all $t$ sufficiently large, we have
 \begin{equation}\label{ineq:quadratic convergence}
  \left\|x_{t+1}-x^*\right\| \leq D \|x_t- x^*\|^2.
 \end{equation}
\end{theorem}
\begin{proof}
Let $\tilde g_t:=\sum_{j \in \J}(\lambda_{t+1})_j\nabla F_j(x_t)$. Hence,   \eqref{eq:Deltaxt-lambdanablaFxt}   with $\lambda=\lambda_{t+1}$ implies that $\|\Delta x_{t+1}\|\leq (4/\mu)\|\tilde g_t\|$. Thus,  in view of Lemma~\ref{lem3:basic-gt+1<Deltax},  the facts that $\sigma_t\leq \sigma_{max}$  and  $E^G_j(x_t)=0$ and $E^H_j(x_t)=0$ (See the assumptions of the theorem and \eqref{definition-approxGrad-Hessian}), we have  
\begin{equation*}
\left\|g_{t+1}\right\|\leq \frac{L+\sigma_{max}+2\theta}{2}\|\Delta x_{t+1}\|^2\leq \frac{16}{\mu^2}\left(\frac{L+\sigma_{max}+2\theta}{2}\right)\|\tilde g_t\|^2.
\end{equation*}
Hence, defining $\tilde C:=8(L+\sigma_{max}+2\theta)/\mu^2$, we have
$ \|g_{t+1}\|\leq {\tilde C}\|\tilde g_t\|^2.$
Thus, it follows from     \eqref{ineq:xi-xt<gt+1},  \eqref{aux-eq: tilde gt}, and the latter inequality that 
$$
\|x_{t+1}-x^*\|\leq C_1\|g_{t+1}\|\leq C_1\tilde C\|\tilde g_t\|^2\leq C_1\tilde C C_2^2\|x_t-x^*\|^2.
$$
Therefore, \eqref{ineq:quadratic convergence} follows by defining $D:=C_1\tilde C C_2^2$.
\end{proof}

\subsection{Approximations of the gradients and Hessians using finite differences}\label{sec:approximation}

In this section, we discuss implementations of the M-CRM where the approximations of the gradient and Hessian of each function component \( F_j \) are computed using finite difference procedures; see, for example,  \cite{nocedal,Geovanni23,Geovanizero}.  We provide bounds for the derivative approximations, whose proofs are given in the Appendix, to ensure the completeness of our presentation.
  { 
It is worth noting that in these implementations, the constants \(\bar \kappa_{G}\) and \(\bar \kappa_{H}\)  depend on the Lipschitz constants \( L_j \). However, knowing the exact values of these constants are not necessary, as the first and/or second  conditions in \eqref{ineq:ApproxGrad-Hess} are trivially satisfied.
}

\subsubsection{Approximations for the gradient using finite differences}

In this subsection, we discuss sufficient conditions under which the first condition in \eqref{ineq:ApproxGrad-Hess} is met by suitable finite-difference approximations of \(\nabla F_j\). We begin by discussing the approximations of \(\nabla F_j\) using central finite differences under hypothesis {\bf (A1)}.

\begin{lemma}
\label{lem:2.378}
Suppose that {\bf (A1)} holds. Given $x\in\mathbb{R}^{n}$ and $h>0$, let $\nabla \bar F_j(x)\in\mathbb{R}^{n}$ be defined by
\begin{equation}
\nabla \bar F_j(x)=\left(\dfrac{ F_j(x+he_{1})- F_j(x-he_{1})}{2h},\ldots,\dfrac{ F_j(x+he_{n})- F_j(x-he_{n})}{2h}\right), \quad \forall j \in \J.
\label{eq:2.147814}
\end{equation}
Then, for every $j \in \J$,  we have
\begin{equation}
\|E^G_{j}(x)\|\leq\dfrac{\sqrt{n}L_j}{6}h^2,
\label{eq:2.17mais1}
\end{equation}
where $E^G_{j}(x)$ is as  in \eqref{definition-approxGrad-Hessian}.
As a consequence, by taking $h:={\sqrt{6} \|x_t-x_{t-1}\|}/(n^{\beta}\alpha^{(i-1)})^{1/2}$, with $\beta \geq 0$, we have   $\nabla \bar F_j(x_t)$  satisfies the first inequality in \eqref{ineq:ApproxGrad-Hess} with $\bar \kappa_{G}:=Ln^{(1-2\beta)/2}$, where $L:=\max_{j\in \J} L_j$. 
\end{lemma}

We now discuss  suitable finite-difference approximations of $\nabla F_j$ under an additional hypothesis, which trivially holds when the gradient $\nabla F_j$ is Lipschitz continuous.
The first and second approaches use forward and backward finite differences, while the third employs central finite differences.

\begin{lemma} \label{lem:2.37890}
Assume that, for every $j\in \J$,   there exists $\bar L_j>0$ such that 
\begin{equation}\label{eq:o90}
|F_j(y)- F_j(x)-\langle\nabla F_j(x),y-x\rangle|\leq 
\frac{\bar L_j}{2} \|y - x\|^2, \quad  \forall  x, y \in \R^n.
\end{equation}
 Given $x\in\mathbb{R}^{n}$ and $h>0$, let $\nabla \bar F_j(x)\in\mathbb{R}^{n}$ be defined, for every $ j \in \J$, by
\begin{equation}
\nabla \bar F_j(x)=\left(\dfrac{ F_j(x+he_{1})- F_j(x)}{h},\ldots,\dfrac{ F_j(x+he_{n})- F_j(x)}{h}\right),
\label{eq:2.14781o0}
\end{equation}
or
\begin{equation}
\nabla \bar F_j(x)=\left(\dfrac{ F_j(x)-F_j(x-he_{1})}{h},\ldots,\dfrac{ F_j(x)-F_j(x-he_{n})}{h}\right),
\label{eq:2.14781o01}
\end{equation}
or 
\begin{equation}
\nabla \bar F_j(x)=\left(\dfrac{ F_j(x+he_{1})- F_j(x-he_{1})}{2h},\ldots,\dfrac{ F_j(x+he_{n})- F_j(x-he_{n})}{2h}\right).
\label{eq:2.147814ji}
\end{equation}
Then, for every $j \in \J$,  we have
\begin{equation}
\|E^G_{j}(x)\|\leq\dfrac{\sqrt{n}\bar L_j}{2}h,
\label{eq:2.17mais2345}
\end{equation}
where $E^G_{j}(x)$ is as  in \eqref{definition-approxGrad-Hessian}. As a consequence, by taking $h:={2 \|x_t-x_{t-1}\|^2}/({n^{\beta}\alpha^{(i-1)}})$,  with $\beta \geq 0$, we have   $\nabla \bar F_j(x_t)$  satisfies the first inequality in \eqref{ineq:ApproxGrad-Hess} with $\bar \kappa_{G}:=\bar Ln^{(1-2\beta)/2}$, where $\bar L:=\max_{j\in \J} \bar L_j$. 
\end{lemma}

{
\begin{remark}
The power \(\beta\) that appears in the definition of \(h\) in the  lemmas of this section can be used to prevent \(h\) from becoming too small, thereby avoiding instability in the finite-difference approximation of \(\nabla F_j\).
\end{remark}
}

\subsubsection{Approximations for the Hessian using finite differences}
Here, we discuss sufficient conditions  under which the second condition in \eqref{ineq:ApproxGrad-Hess} is satisfied by suitable finite-difference approximations of $\nabla^2 F_j$.
We begin by discussing the approximations of \(\nabla^2 F_j\) using  finite differences based on gradient values.

\begin{lemma}
\label{lem:2.3}
Suppose that {\bf (A1)} holds. Given $x\in\mathbb{R}^{n}$ and $h>0$, let $A_j\in\mathbb{R}^{n\times n}$ be defined, for every $j \in \J$, by
\begin{equation}
A_j=\left[\dfrac{\nabla F_j(x+he_{1})-\nabla F_j(x)}{h},\ldots,\dfrac{\nabla F_j(x+he_{n})-\nabla F_j(x)}{h}\right],
\label{eq:2.14}
\end{equation}
or 
\begin{equation}
A_j=\left[\dfrac{\nabla F_j(x)-\nabla F_j(x-he_{1})}{h},\ldots,\dfrac{\nabla F_j(x)-\nabla F_j(x-he_{n})}{h}\right],
\label{eq:2.141}
\end{equation}
or
\begin{equation}
A_j=\left[\dfrac{\nabla F_j(x+he_{1})-\nabla F_j(x-he_{1})}{2h},\ldots,\dfrac{\nabla F_j(x+he_{n})-\nabla F_j(x-he_{1})}{2h}\right].
\label{eq:2.142}
\end{equation}
Then, for every $j \in \J$,  the matrix
\begin{equation}
\nabla^2 \bar F_j(x):=\dfrac{1}{2}\left(A_j+A_j^{T}\right)
\label{eq:2.16}
\end{equation}
satisfies
\begin{equation}
\|E^H_{j}(x)\|\leq\dfrac{\sqrt{n}L_j}{2}h,
\label{eq:2.17mais}
\end{equation}
where $E^H_{j}(x)$ is as  in \eqref{definition-approxGrad-Hessian}. As a consequence, by taking $h:={2 \|x_t-x_{t-1}\|}/{(n^\beta\alpha^{i-1})}$,   with $\beta \geq 0$, we have   $\nabla^2 \bar F_j(x_t)$  satisfies the second inequality in \eqref{ineq:ApproxGrad-Hess} with $\bar \kappa_{H}:=Ln^{(1-2\beta)/2}$, where $L:=\max_{j\in \J} L_j$. 
\end{lemma}

{
\begin{remark} \label{rem2}
We note that if the gradients are exact (in which case we can consider \(\bar{\kappa}_{G} = 0\)) and the Hessians are approximated as in Lemma~\ref{lem:2.3}, it follows from Theorem~\ref{thm:sumNablaF_j<-Delta0} that the M-CRM requires at most \(\mathcal{O}\left(n^{3(1-2\beta)/4}\varepsilon^{-\frac{3}{2}}\right)\) iterations to generate an \(\varepsilon\)-approximate Pareto critical point for problem~\eqref{vectorproblem}, with $\beta \in [0,1/2]$. Additionally, since each inner iteration in this case involves \((n+2)m\) function and gradient evaluations, it follows from Lemma~\ref{lem:3.1} that the total number \(\delta_{T}\) of evaluations of \(F_j(\,\cdot\,)\) and \(\nabla F_j(\,\cdot\,)\) required to generate an \(\varepsilon\)-approximate Pareto critical point is \(\mathcal{O}\left(mn^{(7-6\beta)/4}\varepsilon^{-\frac{3}{2}}\right)\).
Moreover, if  \(m = 1\) (scalar case) and \(\beta = 1/2\), we recover the complexity results of the CRM from \cite[Theorem~2 and Corollary~1]{Geovanni23}.
\end{remark}
}

We next discuss  the approximation of \(\nabla^2 F_j\) using finite differences based on  the function values.

\begin{lemma}
\label{lem:2.356457}
Suppose that {\bf (A1)} holds. Given $x\in\mathbb{R}^{n}$ and $h>0$, let $A_j\in\mathbb{R}^{n\times n}$ be defined, for every $ j \in \J$, by
{\small
\begin{equation}
(A_j)_{il}=\left[\dfrac{ F_j(x+h(e_{i}+e_l))-F_j(x+h(e_{i}-e_l))-F_j(x+h(e_l-e_{i}))+F_j(x-h(e_{i}+e_l))}{4h^2} \right],
\label{eq:2.14899008745}
\end{equation}}
or
\begin{equation}
(A_j)_{il}=\left[\dfrac{ F_j(x+h(e_{i}+e_l))-F_j(x+he_{i})-F_j(x+he_{l})+F_j(x)}{h^2} \right],
\label{eq:2.148990}
\end{equation}
for all $i,l= 1, \ldots,n$. Then, for every $j \in \J$,  the matrix
\begin{equation}
\nabla^2 \bar F_j(x):=\dfrac{1}{2}\left(A_j+A_j^{T}\right)
\label{eq:2.1654}
\end{equation}
satisfies
\begin{equation}
\|E^H_{j}(x)\|\leq\dfrac{2nL_j}{3}h.
\label{eq:2.17mais34}
\end{equation}
where $E^H_{j}(x)$ is as  in \eqref{definition-approxGrad-Hessian}. As a consequence, by taking $h:={3 \|x_t-x_{t-1}\|}/{(2n^\beta \alpha^{i-1})}$,  with $\beta \geq 0$, we have   $\nabla^2 \bar F_j(x_t)$  satisfies the second inequality in \eqref{ineq:ApproxGrad-Hess} with $\bar \kappa_{H}:=Ln^{1-\beta}$, where $L:=\max_{j\in \J} L_j$. 
\end{lemma}

In the following, we provide a derivative-free implementation of the M-CRM.
\\[4mm]
\textbf{Algorithm 2.} A derivative-free variant of the M-CRM 
\\[0.2cm]
\textbf{Step 0.} Choose $x_{0},x_{1}\in\mathbb{R}^{n}$, {$x_0 \ne x_1$,} $\beta \geq 0$, {$\sigma_{1}>0$}, $\theta\geq 0$, $\alpha>1$,  and set  $t:=1$.
\\
{\bf Step 1.}  Find the smallest integer $i\geq 0$ such that $\alpha^{i-1}\sigma_{t}\geq \sigma_{1}$.
\\
{\bf Step 1.1.}  Define
\[ h = \min\left\{\frac{\sqrt{6}}{(n^\beta\alpha^{(i-1)})^{1/2}}, \frac{3}{2n^\beta\alpha^{i-1}}\right\} \|x_t - x_{t-1}\|, \]
and, for every $j \in \J$, compute   $\nabla \bar F_j(x_t)$ and $\nabla^2 \bar F_j(x_t)$ 
as  in Lemmas~\ref{lem:2.378} and~\ref{lem:2.356457}, respectively.
\\
{\bf Step 1.2.} 
Compute an approximate KKT point $(x_{t,i}^{+}, \lambda_{t,i}^{+}) \in \R^n \times \R^m_+$ of subproblem \eqref{newtonproblem} with $x:=x_t$ and $\sigma:= \alpha^i\sigma_t$ 
such that
\small
\[
M_{x_{t},\alpha^{i}\sigma_{t}}(x^{+}_{t,i})\leq 0, \quad \sum_{j=1}^{m} (\lambda^+_{t,i})_j=1, \quad \left \|\sum_{j=1}^{m} (\lambda^+_{t,i})_j\nabla M_{x_t,\alpha^i\sigma_t}^j(x_{t,i}^+)\right\|\leq \theta  \|x_{t,i}^+-x_t\|^2,
\]
where 
\[ M_{x,\sigma}(y):=\max_{j \in \J} \,M^j_{x,\sigma}(y), \quad M^j_{x,\sigma}(y):=\langle \nabla \bar F_j(x),y-x\rangle+\dfrac{1}{2}\langle \nabla^{2}\bar F_j(x) (y-x),y-x\rangle+\dfrac{\sigma}{6}\|y-x\|^{3}.\]
\normalsize
\\
{\bf Step 1.3.} If 
\[ 
F_j(x_{t})-F_j(x^{+}_{t,i})\geq\dfrac{\alpha^{i}\sigma_{t}}{12}\|x^{+}_{t,i}-x_{t}\|^{3}-\dfrac{\sigma_{t}}{12}\|x_{t}-x_{t-1}\|^{3}, \quad \forall j \in \J,
\]
set $i_{t}=i$ and go to Step 2. Otherwise, set $i:=i+1$ and go to Step 1.1.
\\
{\bf Step 2.} Set $x_{t+1}=x^{+}_{t,i_{t}}$, $\lambda_{t+1}=\lambda^+_{t,i}$, $\sigma_{t+1}=\alpha^{i_{t}-1}\sigma_{t}$, $t:=t+1$, and go to Step 1.
\vspace{2mm}

{
\begin{remark}\label{rem:a90}
(i) It follows from  Lemmas~\ref{lem:2.378} and~\ref{lem:2.356457}  that  \( \nabla \bar F_j(x_t) \) and \( \nabla^2 \bar F_j(x_t) \) as chosen in Algorithm~2 satisfy the inequalities in \eqref{ineq:ApproxGrad-Hess} with $\bar \kappa_{G}:=Ln^{(1-2\beta)/2}$ and $\bar \kappa_{H}:=Ln^{1-\beta}$. Hence, for $\beta \in [0,1]$,  we obtain a derivative-free implementation of the M-CRM, with  iteration-complexity bounds  of \(\mathcal{O}(n^{3(1-\beta)/2}\epsilon^{-3/2})\)  to generate an \(\varepsilon\)-approximate Pareto critical point for problem~\eqref{vectorproblem}; see Theorem~\ref{thm:sumNablaF_j<-Delta0}.
Additionally, since each inner iteration of  Algorithm~2 involves $\mathcal{O}(mn^2)$
function evaluations, it follows from Lemma~\ref{lem:3.1} that the total number \(\delta_{T}\) of evaluations of \(F_j(\,\cdot\,)\)  required to generate an \(\varepsilon\)-approximate Pareto critical point is \(\mathcal{O}\left(mn^{(7-3\beta)/2}\varepsilon^{-\frac{3}{2}}\right)\).
Again (see Remark~\ref{rem2}), these bounds imply that for large values of \(\beta\), which make \(h\) close to zero, the number of iterations and function  evaluations decrease, as expected. (ii)  Since Algorithm~2  computes only approximation \( \nabla \bar F(\cdot) \) of the derivative $\nabla F(\cdot)$, an interesting stopping criterion is to consider  $\|\sum_{j\in \J} \lambda_t\nabla \bar F_j(x_t)\|\leq \varepsilon $, for a given tolerance $\varepsilon>0$. First note that this criterion will eventually be satisfied. Indeed, 
 using the triangle inequality, the first observation of item(i) above, the definition of $E^G_j(\cdot)$  in \eqref{definition-approxGrad-Hessian}, and the fact that $\sum_{j\in \J}(\lambda_t)_j=1$, we have
{\small\begin{align}
  \left\|\sum_{j\in \J} (\lambda_t)_j\nabla \bar F_j(x_t)\right\|\leq \left\|\sum_{j\in \J} (\lambda_t)_j\nabla  F_j(x_t)\right\| + \max_{j\in \J} \|E^G_j(x_t)\|  &\leq \left\|\sum_{j\in \J} (\lambda_t)_j\nabla  F_j(x_t)\right\| +  \alpha\bar{\kappa}_G\|x_t-x_{t-1}\|^2.
\end{align}}
The right-hand side of the above inequality converges to zero in view of  Corollary~\ref{cor:conv}, proving that the aforementioned criterion will eventually be satisfied. Additionally, if Algorithm~2 is stopped such that 
\[
\left \|\sum_{j\in \J} (\lambda_t)_j\nabla \bar F_j(x_t)\right\|\leq \varepsilon\quad  \mbox{and} \quad 
\|x_t-x_{t-1}\|^2 \leq \varepsilon,
\]
then  $x_t$ is an  \((1+\alpha\bar{\kappa}_G)\varepsilon\)-approximate Pareto critical point of~\eqref{vectorproblem} in the sense of Definition~\ref{def:aprSol}  due to the fact that
\begin{align*}
  \left\|\sum_{j\in \J} (\lambda_t)_j\nabla F_j(x_t)\right\| &\leq \left\|\sum_{j\in \J} (\lambda_t)_j\nabla \bar F_j(x_t)\right\| +  \left\|\sum_{j\in \J}(\lambda_t)_jE^G_j(x_t)\right\|\\
  &\leq \varepsilon + \alpha\bar{\kappa}_G\|x_t-x_{t-1}\|^2\leq (1+\alpha\bar{\kappa}_G)\varepsilon.
\end{align*}
\end{remark}}

\section{Numerical experiments}\label{sec:numerical}

{ In order to assess the numerical performance of the proposed algorithm, we considered 44 unconstrained multiobjective test problems, encompassing both convex and nonconvex cases with varying numbers of variables \(n\) and objectives \(m\). Table~\ref{tab:testp} provides an overview of these problems; for further details, see, for example, \cite{gonccalves2021globally}.} 
We remark that the bounds $\ell$ and $u$ were only used to generate starting points in a box defined by $\ell \leq x \leq u$:
\begin{equation}\label{eq:x0}
x_0 = (1 - \eta) \ell + \eta u,    
\end{equation}
where $\eta$ is drawn from a uniform distribution in (0,1).

{For numerical stability, we consider a scaled version of problem \eqref{vectorproblem} where 
each $F_j(x)$ is replaced by $\gamma_j F_j(x)$ with $\gamma_j = 1/\max (1, \| \nabla F_j(x_0) \|_{\infty}), j=1,\dots,m$. We remark that the set of Pareto critical points is the same for the original and the scaled problem.}

Our experiments were implemented in Fortran 90 and ran on an Apple M4, 24GB with macOS Sequoia. For solving  subproblem \eqref{newtonproblem}, we used Algencan \cite{andreani2008, algencan} which implements an augmented Lagrangian method with safeguards for general nonlinear programming. For the common parameters of Algorithms~1 and 2, we set $\sigma_1 = 2 \times 10^{-2}$, $\alpha = 2$, $x_0$ from \eqref{eq:x0} and $x_1 = x_0 + 10^{-4} \times \sqrt{n}$.  For Algorithm~2, we use   $\beta = 1/2$.

\begin{table}
\centering
\begin{tabular}{lrrcrr}
\hline 
Problem & n & m & Convex & $\ell$ & $u$ \\   \hline 
AP1 & 2 & 3 & Y & (-10, -10) & (10, 10) \\  
AP2 & 1 & 2 & Y & -100 & 100\\  
AP3 & 2 & 2 & N & (-100, -100) & (100, 100)\\  
AP4 & 3 & 3 & Y & (-10, -10, -10) & (10, 10, 10)\\  
BK1 & 2 & 2 & Y & (-5, -5) & (10, 10)\\  
DD1a & 5 & 2 & N & (-20, . . . , -20) & (20, . . . , 20)\\  
DGO1 & 1 & 2 & N & -10 & 13\\  
Far1 & 2 & 2 & N & (-1, -1) & (1, 1)\\  
FDS & 5 & 3 & Y & (-2, . . . , -2) & (2, . . . , 2)\\  
FF1 & 2 & 2 & N & (-1, -1) & (1, 1)\\  
Hil1 & 2 & 2 & N & (0, 0) & (1, 1)\\  
IKK1 & 2 & 3 & Y & (-50, -50) & (50, 50)\\  
JOS1 & 100 & 2 & Y & (-100, . . . , -100) & (100, . . . , 100)\\  
KW2 & 2 & 2 & N & (-3, -3) & (3, 3)\\  
LE1 & 2 & 2 & N & (-5, -5) & (10, 10)\\  
Lov1 & 2 & 2 & Y & (-10, -10) & (10, 10)\\  
Lov3 & 2 & 2 & N & (-20, -20) & (20, 20)\\  
Lov4 & 2 & 2 & N & (-20, -20) & (20, 20)\\  
Lov5 & 3 & 2 & N & (-2, -2, -2) & (2, 2, 2)\\  
MGH16b & 4 & 5 & N & (-25, -5, -5, -1) & (25, 5, 5, 1)\\  
MGH26b & 4 & 4 & N & (-1, -1, -1 - 1) & (1, 1, 1, 1)\\  
MGH33b & 10 & 10 & Y & (-1, . . . , -1) & (1, . . . , 1)\\  
MHHM2 & 2 & 3 & Y & (0, 0) & (1, 1)\\  
MLF2 & 2 & 2 & N & (-100, -100) & (100, 100)\\  
MMR1c & 2 & 2 & N & (0.1, 0) & (1, 1)\\  
MMR3 & 2 & 2 & N & (-1, -1) & (1, 1)\\  
MOP2 & 2 & 2 & N & (-1, -1) & (1, 1)\\  
MOP3 & 2 & 2 & N & (-$\pi$, -$\pi$) & ($\pi$, $\pi$)\\  
MOP5 & 2 & 3 & N & (-1, -1) & (1, 1)\\  
MOP7 & 2 & 3 & Y & (-400, -400) & (400, 400)\\  
PNR & 2 & 2 & Y & (-2, -2) & (2, 2)\\  
QV1 & 10 & 2 & N & (-5, . . . , -5) & (5, . . . , 5)\\  
SK1 & 1 & 2 & N & -100 & 100\\  
SK2 & 4 & 2 & N & (-10, -10, -10, -10) & (10, 10, 10, 10)\\  
SLCDT1 & 2 & 2 & N & (-1.5, -1.5) & (1.5, 1.5)\\  
SLCDT2 & 10 & 3 & Y & (-1, . . . , -1) & (1, . . . , 1)\\  
SP1 & 2 & 2 & Y & (-100, -100) & (100, 100)\\  
SSFYY2 & 1 & 2 & N & -100 & 100\\  
Toi4b & 4 & 2 & Y & (-2, -2, -2, -2) & (5, 5, 5, 5)\\  
Toi8b & 3 & 3 & Y & (-1, -1, -1, -1) & (1, 1, 1, 1)\\  
Toi9b & 4 & 4 & N & (-1, -1, -1, -1) & (1, 1, 1, 1)\\  
Toi10b & 4 & 3 & N & (-2, -2, -2, -2) & (2, 2, 2, 2)\\  
VU1 & 2 & 2 & N & (-3, -3) & (3, 3)\\  
ZLT1 & 10 & 5 & Y & (-1000, . . . , -1000) & (1000, . . . , 1000)\\  \hline 
\end{tabular}
\caption{Test problems}\label{tab:testp}
\end{table}

\subsection{Exact derivatives}\label{sec:numexact}

{In the first set of experiments we provide analytic expressions for the gradient and Hessian of the objectives, such that $\bar{\kappa}_G = \bar{\kappa}_H = 0$. We consider two versions of Algorithm~1: one that finds a KKT point in Step~1.2 ``exactly'', i.e, $\theta := 10^{-8} / \norm{x_{t,i_t}^{+} - x_t}^2$, and another where the subproblems are solved inexactly, with $\theta:=0.9$ in \eqref{eq:subprob-Alg1}. We changed the stopping criteria of the inner solver (Algencan) to meet those in \eqref{eq:subprob-Alg1}. These two versions shall be called M-CRM-E (exact) and M-CRM-I (inexact).}

{Based on Definition~\ref{def:aprSol}, we consider the stopping criterion
\begin{equation}\label{eq:stop1}
\left\| \sum_{j \in \J} \lambda_{t,i_t}^{+} \nabla F_j(x_{t,i_t}^{+}) \right\|^2 \leq  10 \times {\tt eps}^{1/2},    
\end{equation}
where ${\tt eps} = \sqrt{2^{-52}}$ is the machine precision. We also set a maximum number of iterations to 1,000.}

M-CRM-E and M-CRM-I were compared with a Multiobjective Newton Method with safeguards (M-NS) from \cite{gonccalves2021globally} on the 44 problems listed in Table~\ref{tab:testp} for 100 different starting points (resulting in 4,400 instances). {The default parameters of the M-NS defined in \cite{gonccalves2021globally} were used. We remark that M-NS implementation\footnote{https://github.com/lfprudente/newtonMOP} also employs Algencan for solving the subproblems.} 

Performance profiles of outer iterations, function evaluations, and running time are reported in Figure~\ref{fig:ppex}. 
We observe that both versions of M-CRM are more efficient in terms of outer iterations than M-NS. When it comes to robustness, given the budget of 1,000 iterations, M-NS solved 97.25\% of the instances whereas M-CRM-E and M-CRM-I solved 99.65\% and 99.63\%, respectively. 
{When it comes to function evaluations, the profile follows almost the same trend as that of outer iterations, which indicates that M-CRM-I and M-CRM-E are not demanding an excessive number of backtrackings in Step~1.3. In terms of running time, M-CRM-I is more efficient than both M-NS and M-CRM-E for this set of problems/instances. It is interesting to note that, in general, even though M-CRM-E requires less iterations than M-NS, the iterations of the former are more expensive, which has a negative impact in the total running time. M-CRM-I mitigates this higher cost per iteration by allowing the subproblems to be solved inexactly.}

\begin{figure}
    \centering
    \includegraphics[trim={20 0 30 0},clip,width=0.31\linewidth]{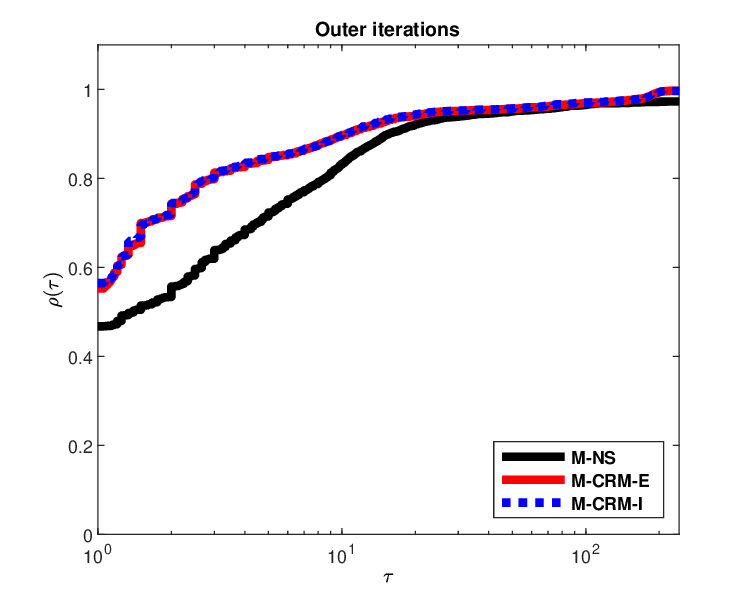}
    \includegraphics[trim={20 0 30 0},clip,width=0.31\linewidth]{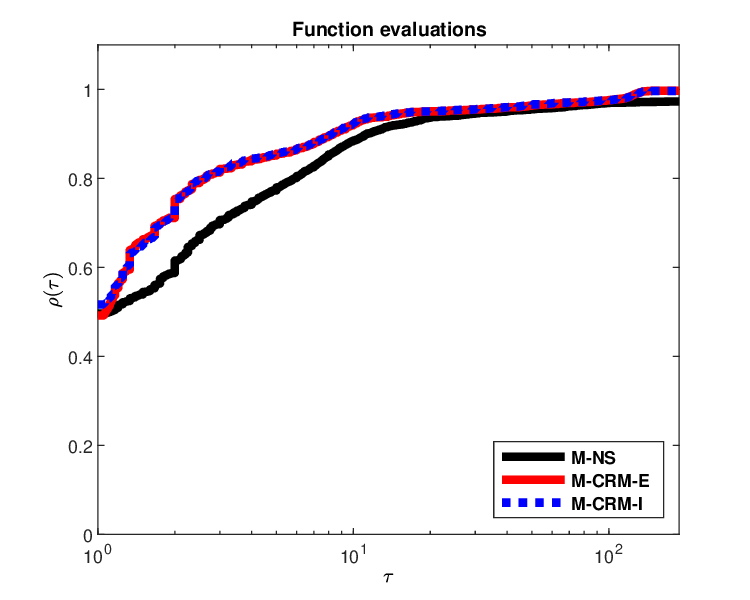}
    \includegraphics[trim={20 0 30 0},clip,width=0.31\linewidth]{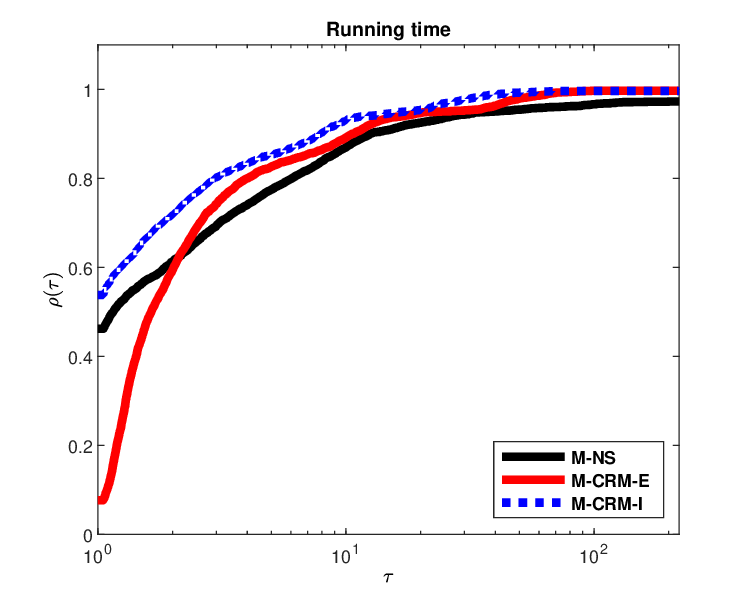}
    \caption{From left to right: outer iterations, function evaluations and running time.}
    \label{fig:ppex}
\end{figure}

\subsection{Finite difference approximations}\label{sec:numinexact} 
Now, we consider a derivative-free version of M-CRM-I given by Algorithm~2, where the gradients and Hessians are approximated by finite differences. We will refer to this version as M-CRM-I-DF. 

{Since M-CRM-I-DF is designed for problems where the derivatives are not readily available, we used as stopping criterion
\begin{equation}\label{eq:stop2}
\left\| \sum_{j \in \J} \lambda_{t,i_t}^{+} \nabla \bar{F}_j(x_{t,i_t}^{+}) \right\| \leq \varepsilon \quad \text{and} \quad      \norm{x_{t,i_t}^{+} - x_t}^2 \leq \varepsilon 
\end{equation}
where $\nabla \bar{F}_j(x)$ denotes the finite difference approximation of $\nabla F_j(x)$ given in Lemma~\ref{lem:2.378} and $\varepsilon^2 = 10 \times {\tt eps}^{1/2}$; see discussion in Remark~\ref{rem:a90}(ii).}

\begin{figure}
\centering
    \includegraphics[trim={20 0 30 0},clip,width=0.31\linewidth]{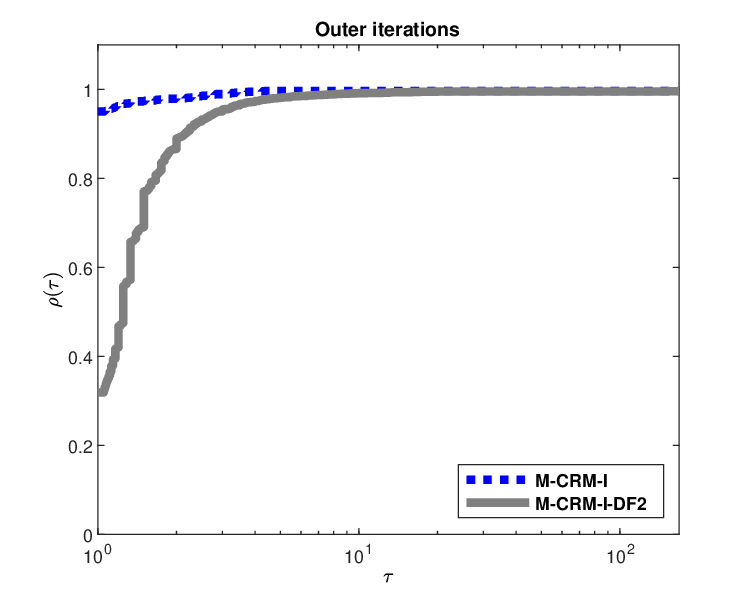}
    \includegraphics[trim={20 0 30 0},clip,width=0.31\linewidth]{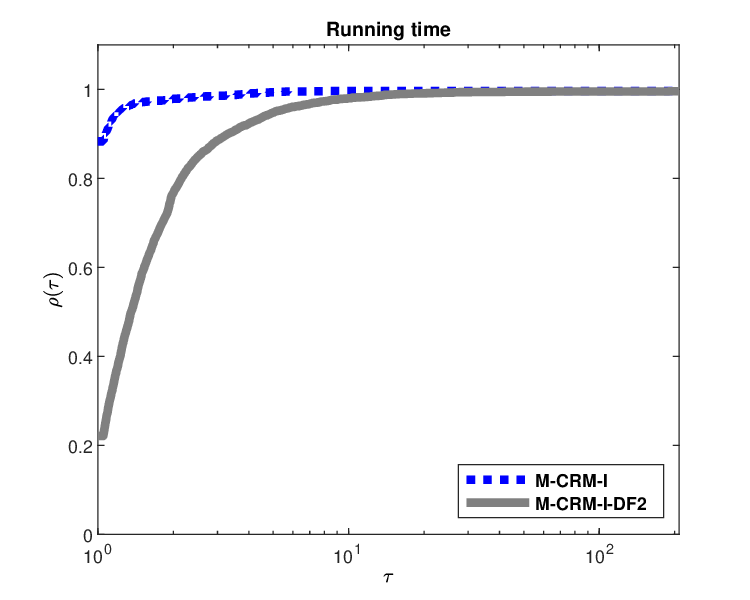}
     \caption{{Outer iterations (left) and running time (right). Performance profiles comparing M-CRM-I and M-CRM-I-DF.}}
    \label{fig:ppex-df}
\end{figure}


Figure~\ref{fig:ppex-df} presents the performance profiles considering outer iterations and running time for M-CRM-I and M-CRM-I-DF. As expected, the use of finite-difference approximations for the derivatives leads to an increase in the number of outer iterations and running time, which is observed in these performance profiles. {However, what is more important to observe is that the robustness was not compromised and M-CRM-I-DF solved 99.54\% of the instances within the iteration limit, very close to the 99.63\% of M-CRM-I.}

Of course, the use of derivative approximations may lead to different iterates and consequently different approximate Pareto points. Hence, it is interesting to investigate whether these two algorithms are recovering the Pareto front in a similar way. Figure~\ref{fig:paretofront} presents the recovered Pareto front for problems BK1, DGO1, Hil1 and Toi4. {Pictures on the left correspond to M-CRM-I which uses exact derivatives whereas those on the right correspond to M-CRM-I-DF. Visually the recovered fronts are quite similar.} 

\begin{figure}
    \centering
    \includegraphics[width=0.45\linewidth, trim={0 40 0 40},clip]{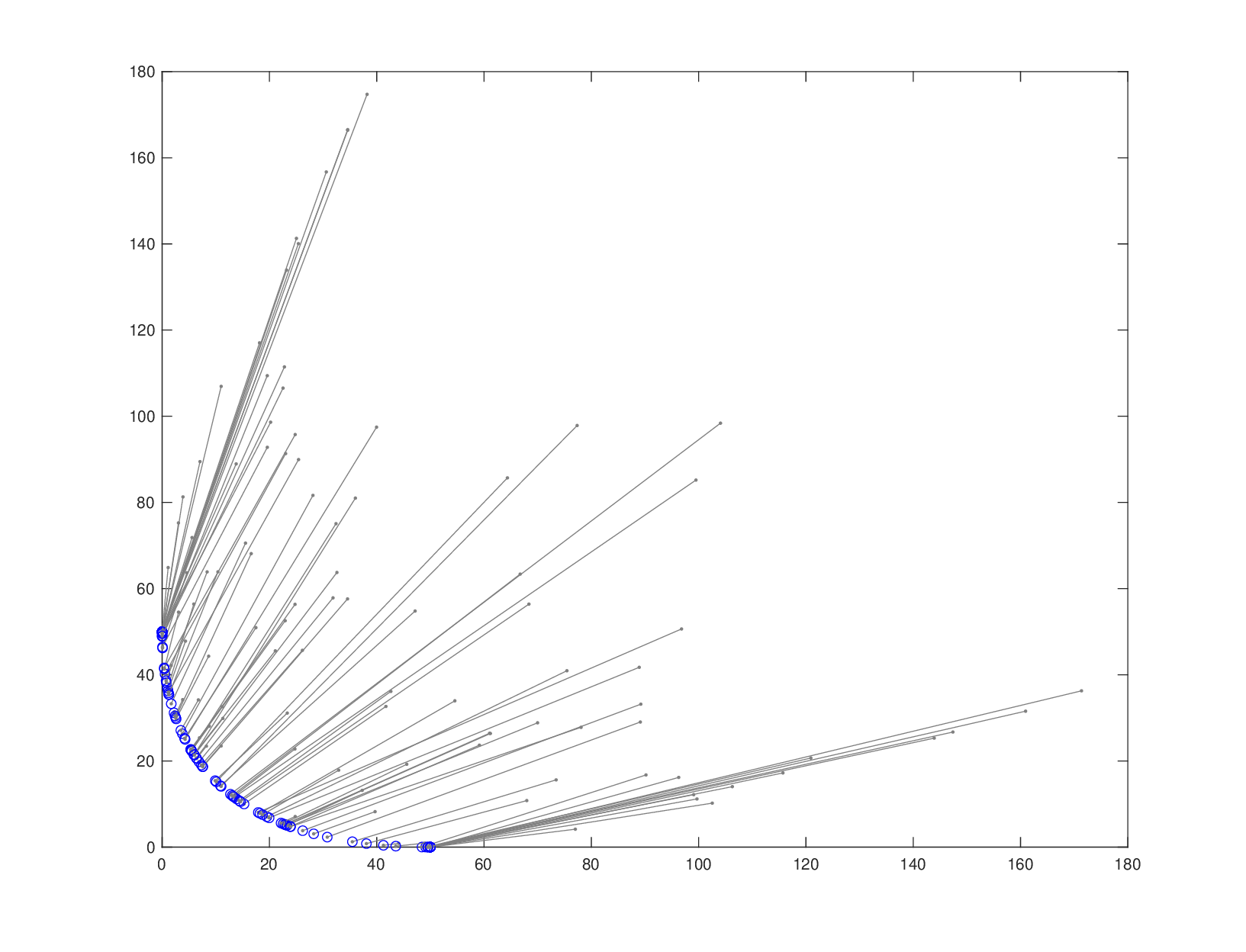}
    \includegraphics[width=0.45\linewidth, trim={0 40 0 40},clip]{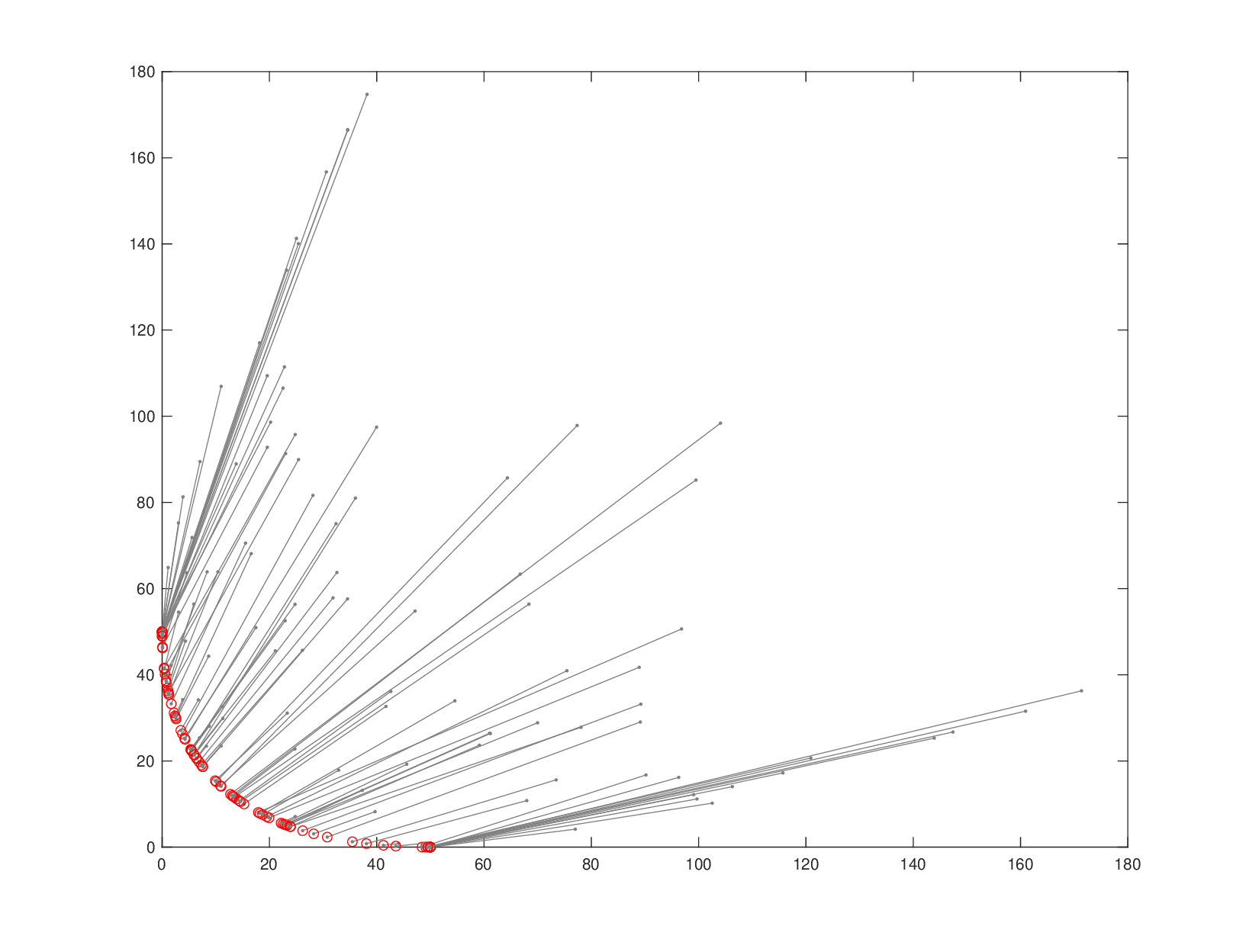}
    \includegraphics[width=0.45\linewidth, trim={0 20 0 20},clip]{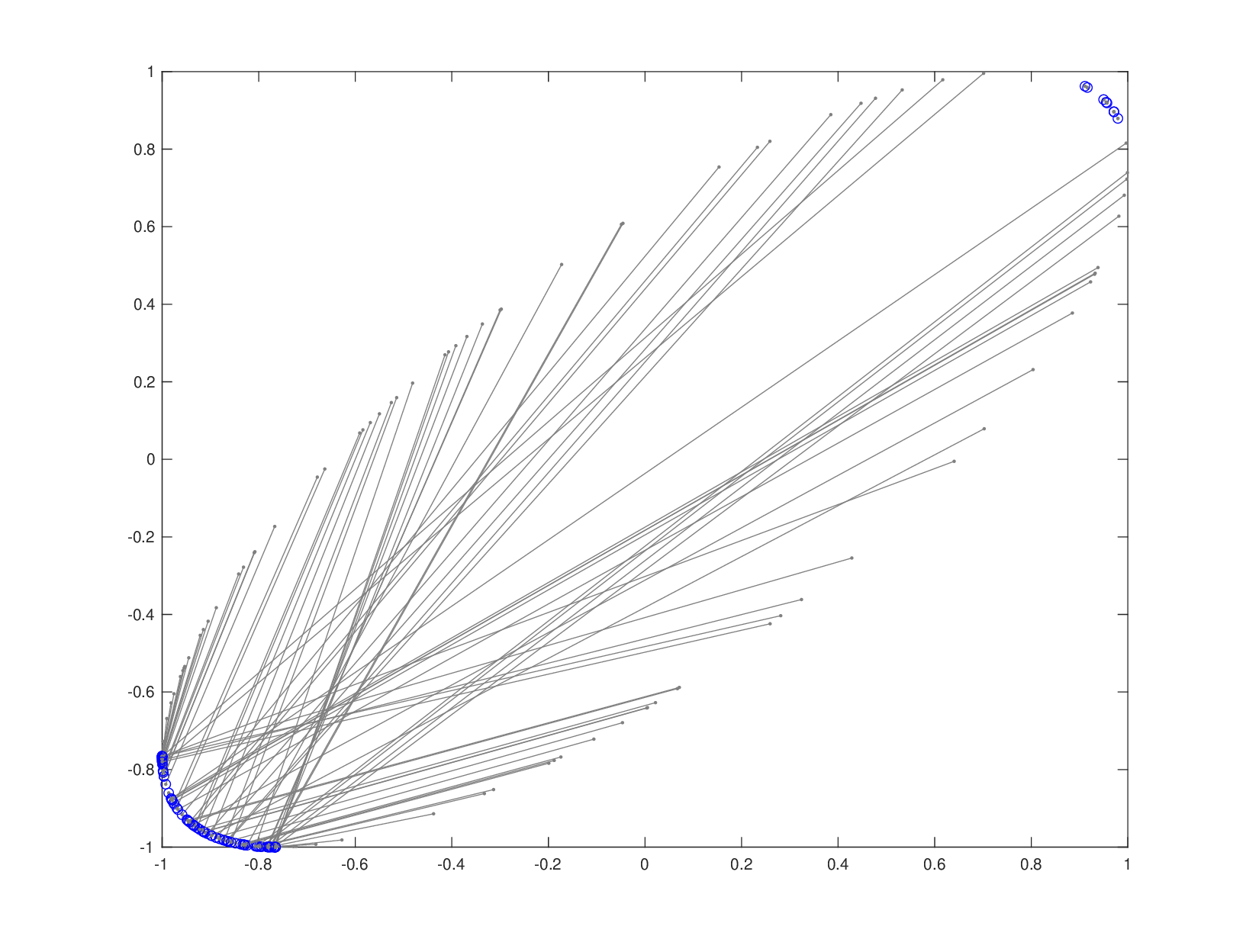}
    \includegraphics[width=0.45\linewidth, trim={0 20 0 20},clip]{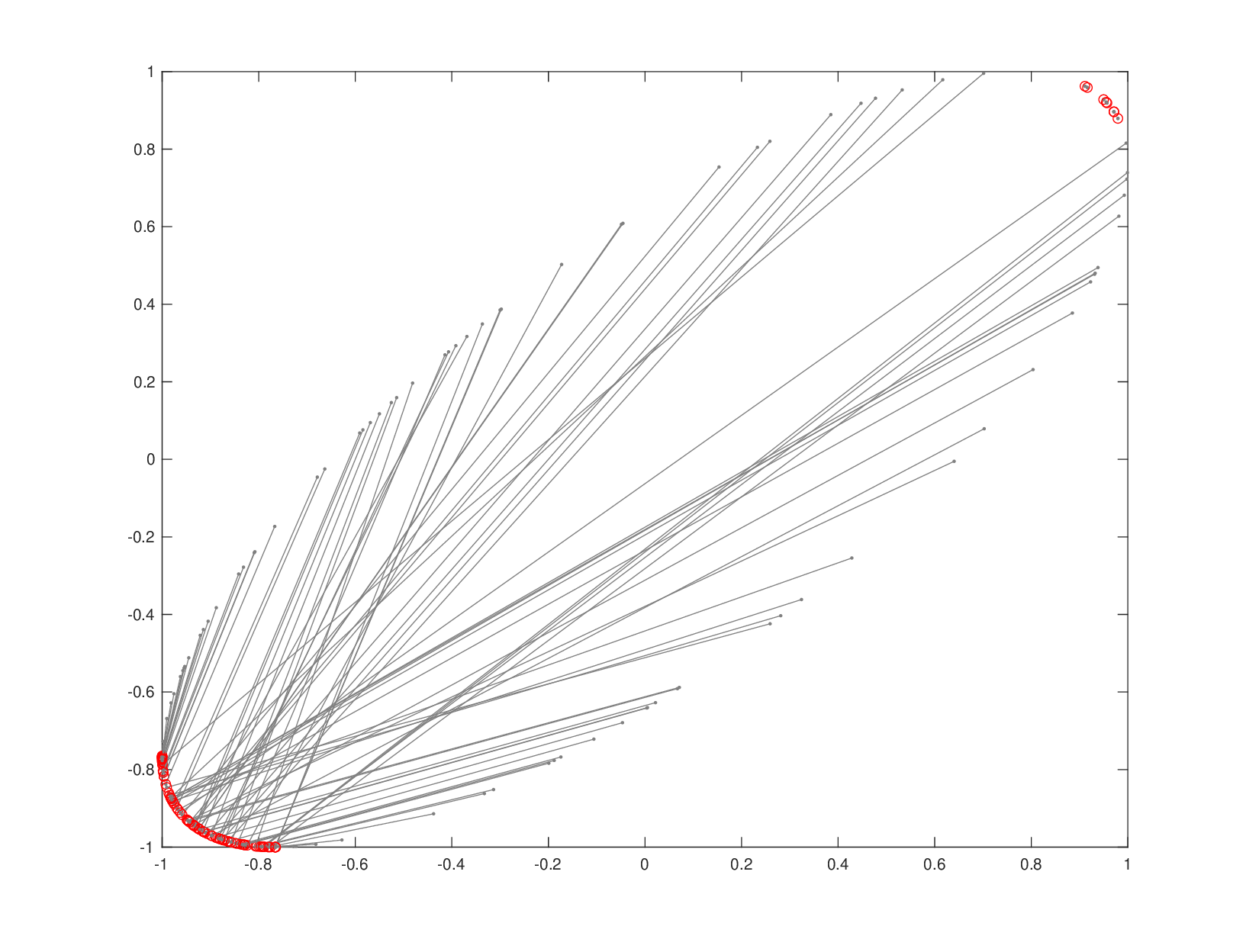}
    \includegraphics[width=0.45\linewidth, trim={0 20 0 20},clip]{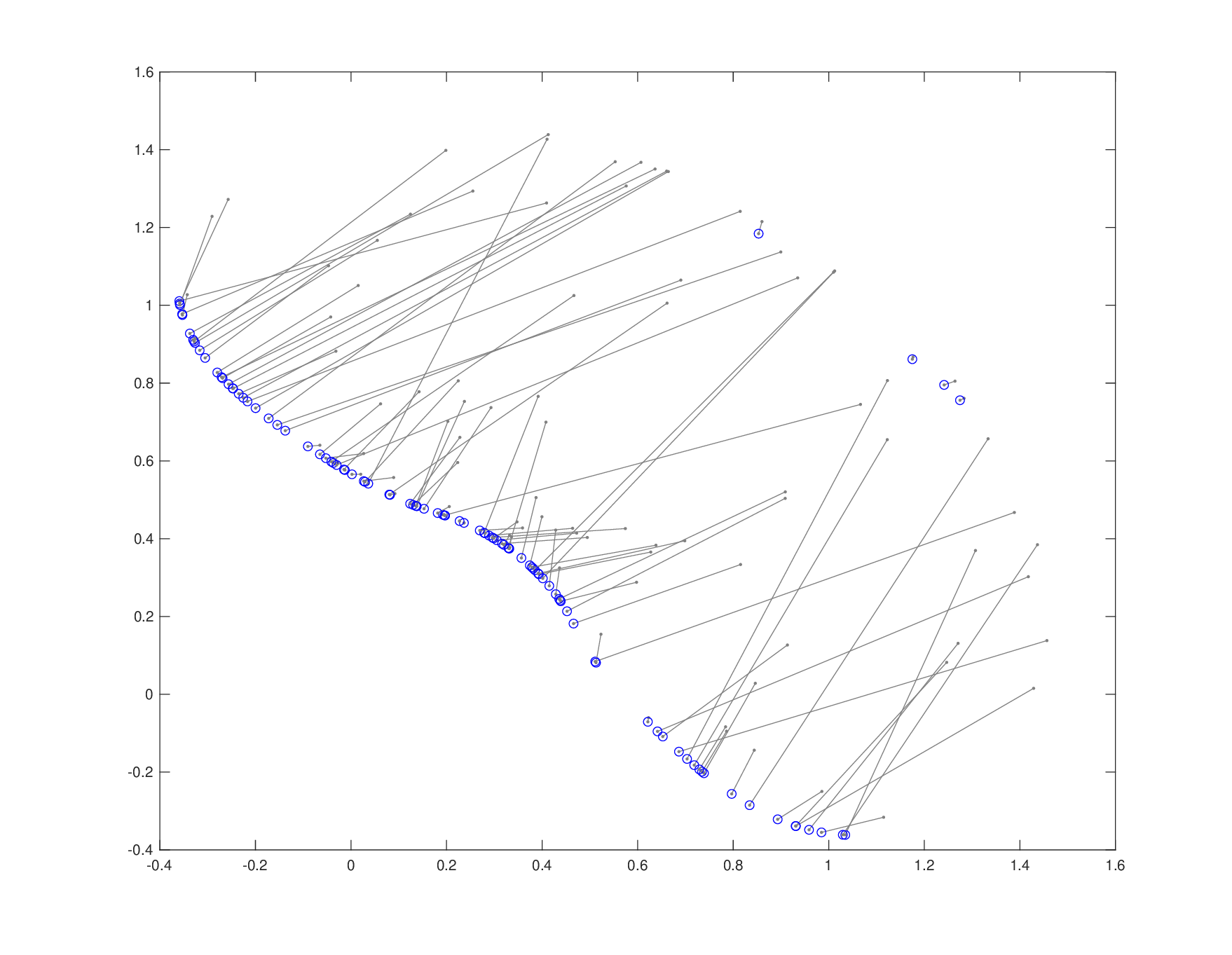}
    \includegraphics[width=0.45\linewidth, trim={0 20 0 20},clip]{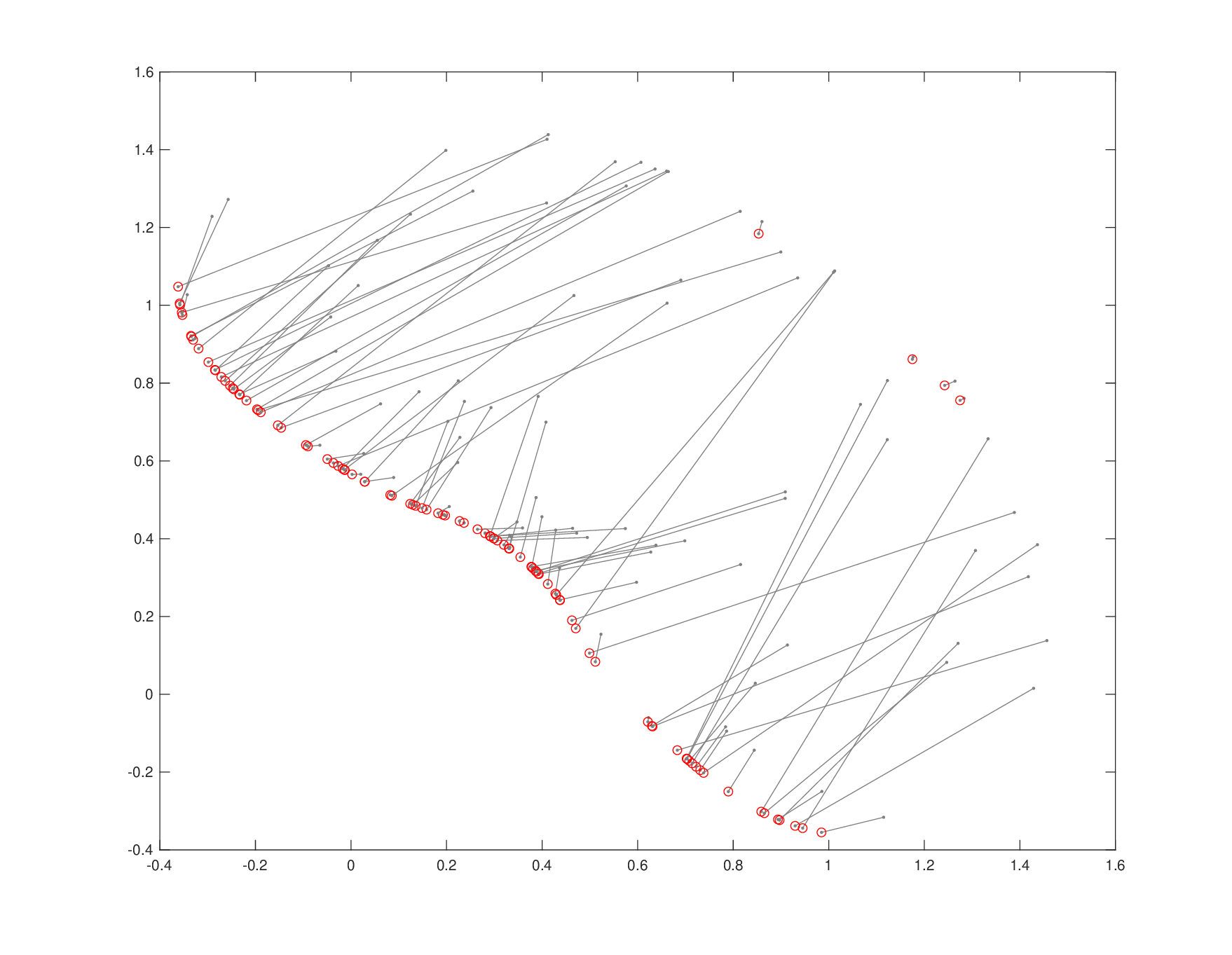}
    \includegraphics[width=0.45\linewidth, trim={0 20 0 20},clip]{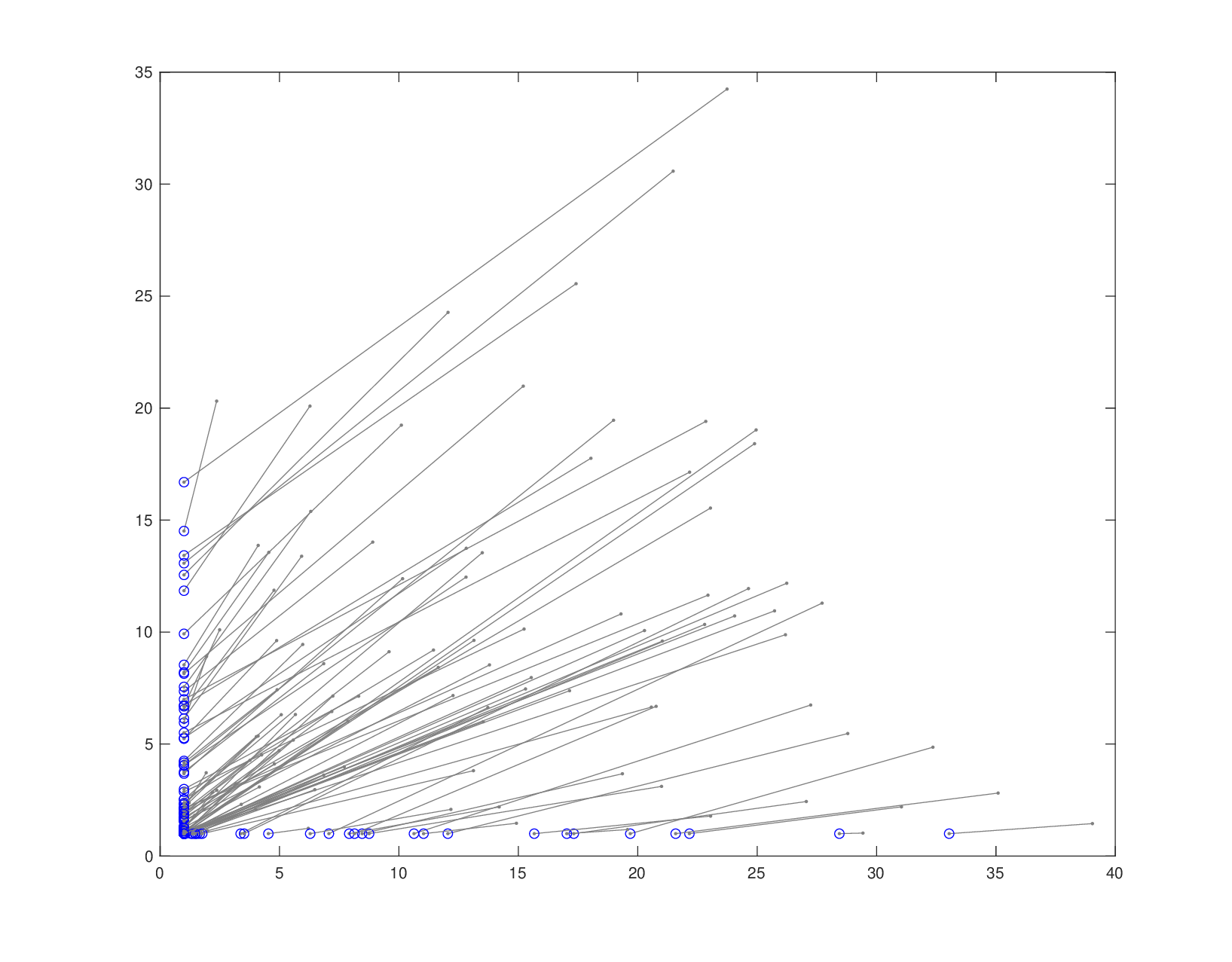}
    \includegraphics[width=0.45\linewidth, trim={0 20 0 20},clip]{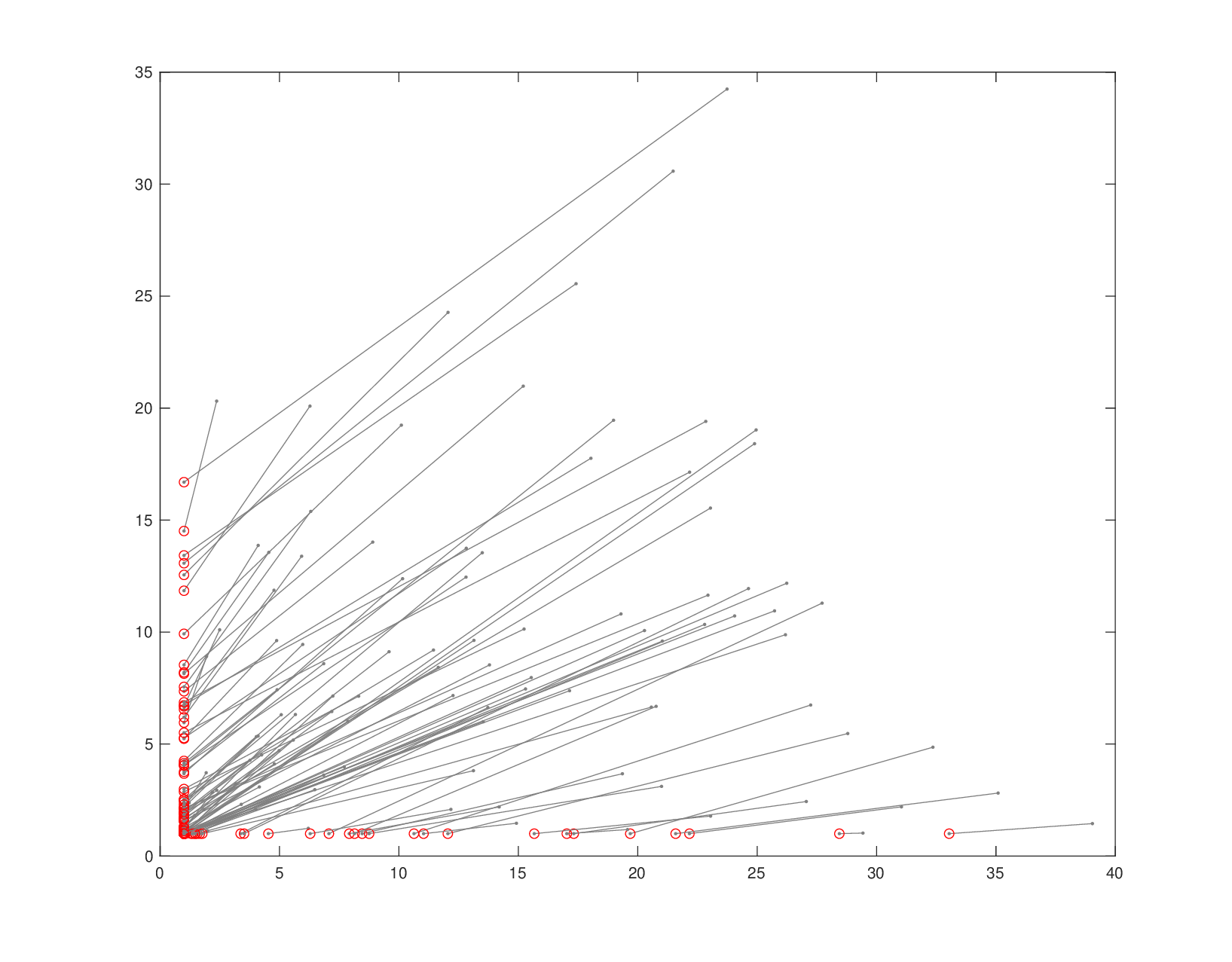}
    \caption{Recovered Pareto front from 100 random starting points. From top to bottom: BK1, DGO1, Hil1, Toi4. Left: M-CRM-I, Right: M-CRM-I-DF.}
    \label{fig:paretofront}
\end{figure}

\subsection{An application to logistic regression}
This last experiment is inspired by regularized logistic regression problems of the form
\begin{equation}\label{eq:logistic}
\min_{x \in \Rn} f_{\mu}(x) := - \sum_{i=1}^m \left[ b^{(i)} \log (\sigma_x(a^{(i)})) + (1 - b^{(i)}) \log (1 - \sigma_x(a^{(i)})) \right] + \frac{\mu}{2} \|x \|_2^2,
\end{equation}
where $\{(a^{(i)},b^{(i)})\}_{i=1}^m \subset \Rn \times \{0,1\}$ is the dataset, $\sigma_x(a):= 1/(1+e^{-\langle a, x \rangle})$ is the logistic model and $\mu > 0$ is a regularization parameter.

Problem \eqref{eq:logistic} can be interpreted as a scalarization for a multiobjective optimization problem (MOP) {with two objectives}
\begin{equation}
    F_1(x) = - \sum_{i=1}^m \left[ b^{(i)} \log (\sigma_x(a^{(i)})) + (1 - b^{(i)}) \log (1 - \sigma_x(a^{(i)})) \right], \quad F_2(x) = \frac{1}{2} \|x \|_2^2.
\end{equation}
Performing analysis of points on the Pareto front of such a MOP allows one to find a good trade-off between goodness of fit (represented by $F_1(x)$) and simplicity of the model (represented by $F_2(x)$). 

{As dataset, we consider the Wisconsin Breast Cancer Dataset\footnote{Dataset freely available in the UCI Machine Learning Repository
(\url{https://archive.ics.uci.edu/ml}).} \cite{street1993}}. We used only the first 10 features and scaled the data so that each feature is in the interval [0,1]. {For scaling the two objectives we considered $\gamma_1 = 0.1$ and $\gamma_2 = 1.0$}. The dataset has 568 records which were split: the first $m=468$ for training and the last 100 for test. 

{We decided to test M-CRM-I-DF for training so we do not need to provide gradients and Hessians for the objectives.} Then, we ran M-CRM-I-DF from 300 random starting points in $\mathbb{R}^{10}$ according to \eqref{eq:x0} with $\ell = (-10,\dots,-10)^\top$ and $u = (10,\dots,10)^\top$.

Figure~\ref{fig:btc} illustrates the recovered Pareto front (and its zoomed version). ``A'' represents the point with smallest training error (smallest $F_1(x))$, but it is the point with the largest Euclidean norm ($F_2(x)$ value). Points up northwest in the Pareto front are more prone to overfitting. Its accuracy in the test set was 93\%.  In the opposite side, point ``C'' is the one with smallest Euclidean norm ($F_2(x) \approx 10^{-14}$), resulting in a model too simple to adjust well the training data. Its test accuracy was only 23\%. In the ``knee'' region of the Pareto front we depict point ``B'' which corresponds to one of the best in terms of test accuracy: 95\%. 
{Even with a value of $F_1(x)$ (training error) slightly higher than that of point ``A'',  the simpler model provided by ``B'' generalized better than the one corresponding to ``A''.}

\begin{figure}
    \centering
    \includegraphics[width=0.45\linewidth]{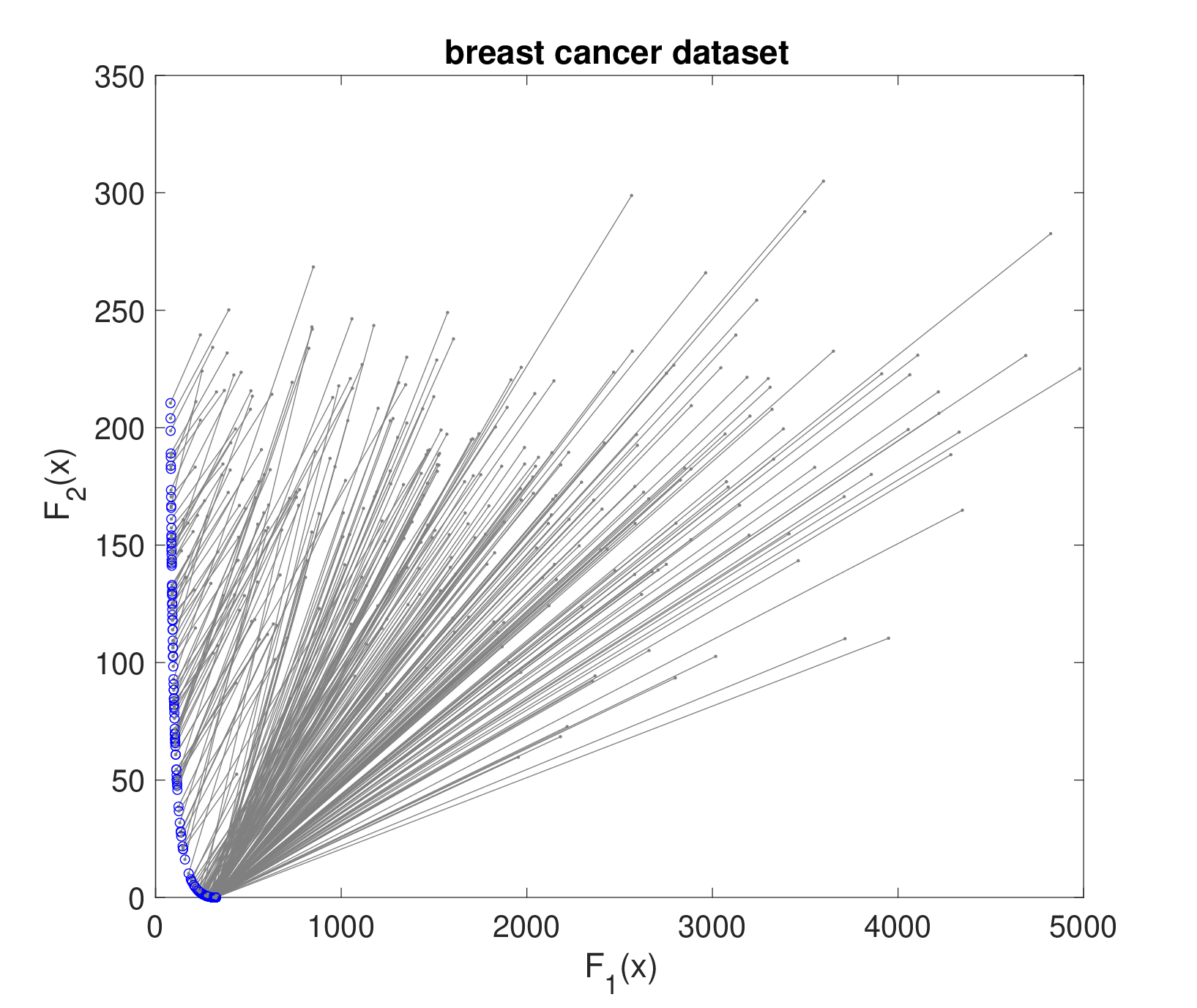}\includegraphics[width=0.45\linewidth]{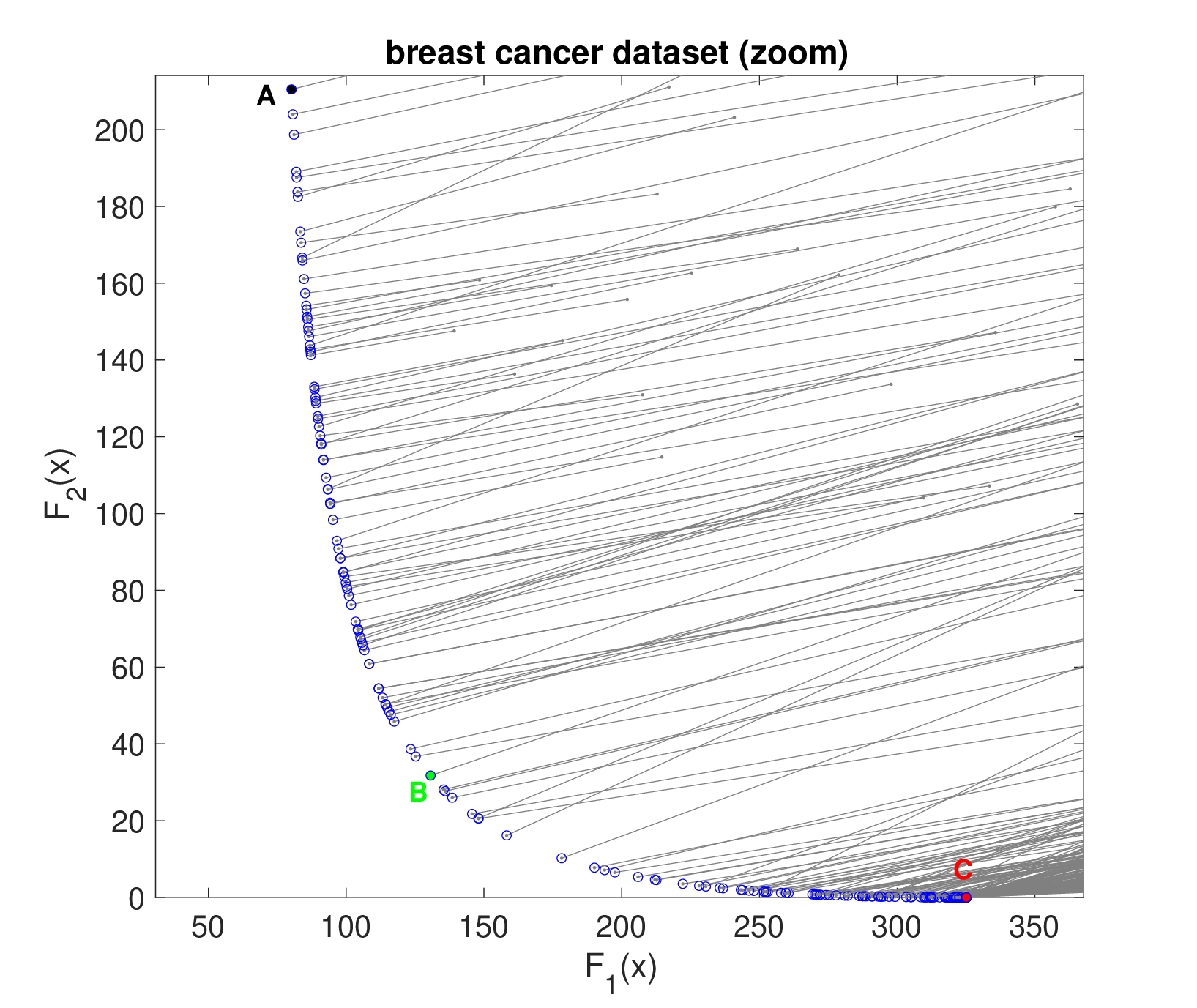}
    \caption{Recovered Pareto front (left) and zoomed front (right). On the right, points A and C indicate the smallest $F_1(x)$ (logistic objective) and $F_2(x)$ ($\|x\|_2$), respectively. Point B corresponds to the best in terms of accuracy in the test set.}
    \label{fig:btc}
\end{figure}




\section{Conclusion}\label{sec:final}

This work proposed a novel cubic regularization method tailored for nonconvex unconstrained multiobjective optimization problems. The method relies on regularized models for each objective component, incorporating inexact first- and second-order derivative information within controllable error bounds. A distinctive aspect of the algorithm is the joint adjustment of the regularization parameter and derivative approximation accuracy through a nonmonotone line search strategy, enhancing its practical robustness. Theoretical analysis established iteration-complexity bounds of \(\mathcal{O}(\epsilon^{-3/2})\) under mild assumptions, with specific results for finite-difference implementations showing favorable complexity in terms of function evaluations. Furthermore, the method exhibits global convergence, and under local convexity, achieves superlinear or quadratic local convergence, depending on the accuracy of the derivative information.

Numerical experiments on standard test sets from the literature showcase the effectiveness of the proposed methods and their competitive performance with respect to a safeguarded multiobjective Newton method. We also remark the robustness of the derivative-free version of M-CRM in these test problems and its applicability in real problems as illustrated by the example in regularized logistic regression.

\appendix

\section{Proof of Lemma~\ref{lem:2.378}}
It follows from the triangle inequality for norms and (\ref{eq:2.2}) that, for every $j\in \J$,
\begin{align*}
 & | F_j(x+he_{i})- F_j(x-he_{i})-2h\inner{\nabla F_j(x)}{e_{i}}|\\
  &\leq| F_j(x+he_{i})- F_j(x)-h\inner{\nabla F_j(x)}{e_{i}}-h^2\inner{\nabla^{2}F_j(x)e_{i}}{e_i}|\\
 &+\left|-F_j(x-he_{i})+F_j(x)-h\inner{\nabla F_j(x)}{e_{i}}+h^2\inner{\nabla^{2}F_j(x)e_{i}}{e_i}\right|\\
 &\leq\dfrac{L_j}{6}h^{3}+\dfrac{L_j}{6}h^{3}=\dfrac{L_j}{3}h^{3},
\end{align*}
which implies that 
\begin{equation*}
 \left|\dfrac{ F_j(x+he_{i})- F_j(x-he_{i})}{2h}-\inner{\nabla F_j(x)}{e_{i}}\right|\leq\dfrac{L_j}{6}h^2.
\end{equation*}
Hence, if  $\nabla \bar F_j$ is as in \eqref{eq:2.147814}, we obtain 
\begin{equation}\label{eq:8901909}
 | ({\nabla \bar F_j(x)-\nabla F_j(x)})_i|\leq\dfrac{L_j}{6}h^2, \quad \forall i=1, \dots, n.
\end{equation}
which yields 
\[
 \|E^G_{j}(x)\|=\| {\nabla \bar F_j(x)-\nabla F_j(x)}\| \leq\dfrac{\sqrt{n}L_j}{6}h^2,
\]
concluding the proof of \eqref{eq:2.17mais1}.

The second statement of the lemma trivially follows from the first one and the definition of $h$ and~$L$.

\section{Proof of Lemma~\ref{lem:2.37890}}
It follows from \eqref{eq:o90} with $y=x+he_{i}$ that, for every $j\in \J$,
\begin{equation*}
| F_j(x+he_{i})- F_j(x)-h\inner{\nabla F_j(x)}{e_{i}}|\leq\dfrac{\bar L_j}{2}h^{2}
\end{equation*}
or, equivalently,
\begin{equation*}
  \left|\left(\dfrac{ F_j(x+he_{i})- F_j(x)}{h}\right)-\inner{\nabla F_j(x)}{e_{i}}\right|\leq\dfrac{\bar L_j}{2}h.
\end{equation*}
Hence, if  $\nabla \bar F_j$ is as in \eqref{eq:2.14781o0}, we obtain 
\begin{equation}\label{eq:89011}
\|({\nabla \bar F_j(x)-\nabla F_j(x)})_i\|\leq\dfrac{\bar L_j}{2}h,  \quad \forall i=1, \dots, n.
\end{equation}
Using similar arguments, it can be proven that this inequality also holds for \(\nabla \bar{F}_j\) as defined in \eqref{eq:2.14781o01}.
On the other hand,  it follows from the triangle inequality for norms and \eqref{eq:o90} that, for every $j\in \J$,
\begin{align*}
 & | F_j(x+he_{i})- F_j(x-he_{i})-2h\inner{\nabla F_j(x)}{e_{i}}|\\
 & \leq| F_j(x+he_{i})- F_j(x)-h\inner{\nabla F_j(x)}{e_{i}}|
+\left|-F_j(x-he_{i})+ F_j(x)-h\inner{\nabla F_j(x)}{e_{i}}\right|\\
 &\leq\dfrac{\bar L_j}{2}h^{2}+\dfrac{\bar L_j}{2}h^{2}={\bar L_j}h^{2},
\end{align*}
which implies that 
\begin{equation*}
 \left|\left(\dfrac{ F_j(x+he_{i})- F_j(x-he_{i})}{2h}\right)-\inner{\nabla F_j(x)}{e_{i}}\right|\leq\dfrac{\bar L_j}{2}h.
\end{equation*}
Hence, if  $\nabla \bar F_j$ is as in \eqref{eq:2.147814ji}, we obtain 
\begin{equation}\label{eq:89019091}
 | ({\nabla \bar F_j(x)-\nabla F_j(x)})_i|\leq\dfrac{\bar L_j}{2}h, \quad \forall i=1, \dots, n. 
\end{equation}
Therefore, for the three choices of $\nabla \bar F_j$, we obtain from \eqref{eq:89011} and \eqref{eq:89019091} that 
\[
\|E^G_{j}(x)\|=\| {\nabla \bar F_j(x)-\nabla F_j(x)}\| \leq\dfrac{\sqrt{n}\bar L_j}{2}h,
\]
concluding the proof of \eqref{eq:2.17mais2345}.

The last statement of the lemma trivially follows from the first one and the definitions of $h$ and~$\bar L$.

\section{Proof of Lemma~\ref{lem:2.3}}
It follows from (\ref{eq:2.3}) with $y=x+he_{i}$ that, for every $j\in \J$,
\begin{equation*}
\|\nabla F_j(x+he_{i})-\nabla F_j(x)-h\nabla^{2}F_j(x)e_{i}\|\leq\dfrac{L_j}{2}h^{2}
\end{equation*}
\begin{equation*}
\Longrightarrow \left\|\left(\dfrac{\nabla F_j(x+he_{i})-\nabla F_j(x)}{h}\right)-\nabla^{2}F_j(x)e_{i}\right\|\leq\dfrac{L_j}{2}h.
\end{equation*}
Hence, if  $A_j$ is as in \eqref{eq:2.14}, we obtain 
\begin{equation}\label{eq:890}
 \|(A_j-\nabla^{2}F_j(x))e_{i}\|\leq\dfrac{L_j}{2}h.
\end{equation}
Using similar arguments, it can be proven that this inequality also holds for $A_j$ as defined in \eqref{eq:2.141}.
On the other hand, 
it follows from the triangle inequality for norms and (\ref{eq:2.3}) that, for every $j\in \J$,
\begin{align*}
 & \|\nabla F_j(x+he_{i})-\nabla F_j(x-he_{i})-2h\nabla^{2}F_j(x)e_{i}\|\\
 & \leq\|\nabla F_j(x+he_{i})-\nabla F_j(x)-h\nabla^{2}F_j(x)e_{i}\|+\|-(\nabla F_j(x-he_{i})-\nabla F_j(x)+h\nabla^{2}F_j(x)e_{i})\|\\
 &\leq\dfrac{L_j}{2}h^{2}+\dfrac{L_j}{2}h^{2}={L_j}h^{2},
\end{align*}
which implies that 
\begin{equation*}
 \left\|\left(\dfrac{\nabla F_j(x+he_{i})-\nabla F_j(x-he_{i})}{2h}\right)-\nabla^{2}F_j(x)e_{i}\right\|\leq\dfrac{L_j}{2}h.
\end{equation*}
Hence, if  $A_j$ is as in \eqref{eq:2.142}, we obtain 
\begin{equation}\label{eq:8901}
 \|(A_j-\nabla^{2}F_j(x))e_{i}\|\leq\dfrac{L_j}{2}h.
\end{equation}
Therefore, for the three choices of $A_j$, we obtain from \eqref{eq:890} and \eqref{eq:8901}
\begin{equation*}
\|A_j-\nabla^{2}F_j(x)\|^{2}\leq\|A_j-\nabla^{2}F_j(x)\|_{F}^{2}=\sum_{i=1}^{n}\|(A_j-\nabla^{2}F_j(x))e_{i}\|_{2}^{2}\leq n\left(\dfrac{L_j}{2}\right)^{2}h^{2},
\end{equation*}
which gives
\begin{equation}
\|A_j-\nabla^{2}F_j(x)\|\leq\dfrac{\sqrt{n}L_j}{2}h.
\label{eq:2.17}
\end{equation}
Finally, combining (\ref{eq:2.16}) and (\ref{eq:2.17}), we get
\begin{equation*}
\|E^H_{j}(x)\|=\|\nabla^2 \bar F_j(x)-\nabla^{2}F_j(x)\|\leq\|A_j-\nabla^{2}F_j(x)\|\leq\dfrac{\sqrt{n}L_j}{2}h,
\end{equation*}
which proves the inequality in \eqref{eq:2.17mais}.

The last  statement of the lemma trivially follows from the first one and the definitions of $h$ and $L$.

\section{Proof of Lemma~\ref{lem:2.356457}}
 Note first that, for every $j\in \J$,  from \eqref{eq:2.14899008745}, the difference $ 4h^2 (A_j)_{il} - 4{h^2} (\nabla^{2}F_j(x))_{il}$ is given by 
\begin{align*}
 &  F_j(x+h(e_{i}+e_l))-F_j(x+h(e_{i}-e_l))-F_j(x+h(e_l-e_{i}))+F_j(x-h(e_i+e_{l}))-4{h^2} (\nabla^{2}F_j(x))_{il}\\
 & =F_j(x+h(e_{i}+e_l)) - F_j(x)-h\inner{\nabla F_j(x)}{e_{i}+e_l}-\frac{h^2}{2}\inner{\nabla^{2}F_j(x)(e_{i}+e_l)}{e_i+e_l}\\
& -\left(F_j(x+h(e_{i}-e_l)) - F_j(x)-h\inner{\nabla F_j(x)}{e_{i}-e_l}-\frac{h^2}{2}\inner{\nabla^{2}F_j(x)(e_{i}-e_l)}{e_i-e_l}\right)\\
& -\left(F_j(x+h(e_l-e_{i})) - F_j(x)-h\inner{\nabla F_j(x)}{e_l-e_{i}}-\frac{h^2}{2}\inner{\nabla^{2}F_j(x)(e_l-e_{i})}{e_l-e_i}\right)\\
 &+F_j(x-h(e_{i}+e_l)) - F_j(x)+h\inner{\nabla F_j(x)}{e_{i}+e_l}-\frac{h^2}{2}\inner{\nabla^{2}F_j(x)(e_{i}+e_l)}{e_i+e_l}.
\end{align*}
Therefore, using the triangle inequality for norms and (\ref{eq:2.2}), we have
\begin{align*}
|4h^2 (A_j)_{il} - 4{h^2} (\nabla^{2}F_j(x))_{il} | &\leq  \frac{1}6L_jh^3\left(\|e_{i}+e_l\|^3+\|e_{i}-e_l\|^3+\|e_{l}-e_i\|^3+\|e_{i}+e_l\|^3\right)\\
&\leq  \frac{8}3L_jh^3.
\end{align*}
For $A_j$ as in \eqref{eq:2.14899008745},
dividing the last inequality by $4{h^2}$, we obtain
\begin{equation}\label{eq:890109}
 |(A_j)_{il}-[\nabla^{2}F_j(x)]_{il}|\leq\dfrac{2L_j}{3}h.
\end{equation}
Using similar arguments, it can be proven that this inequality also holds for $A_j$ as defined in \eqref{eq:2.148990}; see, for example, \cite[Lemma~3.3]{Geovanizero}.
Finally, combining the last inequality and (\ref{eq:2.1654}), we get
\begin{equation*}
\|E^H_{j}(x)\|=\|\nabla^2 \bar F_j(x)-\nabla^{2}F_j(x)\|\leq\|A_j-\nabla^{2}F_j(x)\|\leq n\|A_j-\nabla^{2}F_j(x)\|_{\max}\leq \dfrac{2nL_j}{3}h,
\end{equation*}
which proves the inequality in \eqref{eq:2.17mais34}.

The second statement of the lemma trivially follows from the first one and the definition of $h$ and~$L$.

\bibliographystyle{siamplain}
\bibliography{Referencias}

\end{document}